\documentclass[a4paper,11pt]{amsart}

\usepackage[normalem]{ulem}

\usepackage[hmargin=3cm,top=3cm,bottom=2cm]{geometry}

\usepackage{amsmath,amssymb,amsfonts,amsthm,mathrsfs}
\usepackage{enumitem,xcolor,microtype}

\usepackage[colorlinks=true,linkcolor=black,citecolor=black]{hyperref}

\usepackage[utf8]{inputenc} 
\usepackage[T1]{fontenc}
\usepackage{lmodern}

\newcommand{\A}{\mathbf A}

\newcommand{\D}{\mathbf D}
\newcommand{\R}{\mathbb R}
\newcommand{\N}{\mathbb N}
\newcommand{\T}{\mathbb T}

\newcommand{\DD}{\mathcal D}
\newcommand{\EE}{\mathcal E}

\newcommand{\dd}{\mathrm d}

\def\({\left(}
\def\){\right)}
\def\<{\left\langle}
\def\>{\right\rangle}
\def\le{\leqslant}
\def\ge{\geqslant}
\def\leq{\leqslant}
\def\geq{\geqslant}

\def\Eq#1#2{\mathop{\sim}\limits_{#1\rightarrow#2}}
\def\Tend#1#2{\mathop{\longrightarrow}\limits_{#1\rightarrow#2}}

\DeclareMathOperator{\Div}{div}

\DeclareMathOperator{\IM}{Im}

\newcommand{\eps}{\varepsilon}
\renewcommand{\d}{\mathrm{d}}
\newcommand{\dt}{\mathrm{d}t}
\newcommand{\dx}{\mathrm{d}x}
\newcommand{\dy}{\mathrm{d}y}

\newcommand{\ro}{\rho}
\newcommand{\si}{\sigma}

\theoremstyle{plain}
\newtheorem{theorem}{Theorem} [section]
\newtheorem{lemma}[theorem]{Lemma}

\newtheorem{proposition}[theorem]{Proposition}

\theoremstyle{remark}
\newtheorem{remark}[theorem]{Remark}

\theoremstyle{definition}
\newtheorem{definition}[theorem]{Definition}

\newtheorem*{notations}{Notations}

\numberwithin{equation}{section}

\title[Large-time behavior of compressible fluids]{Large-time behavior of compressible polytropic fluids and nonlinear Schr\"odinger equation}
\author[R. Carles]{R\'emi Carles}
\author[K. Carrapatoso]{Kleber Carrapatoso}
\author[M. Hillairet]{Matthieu Hillairet}

\address[R. Carles]{Univ Rennes, CNRS\\ IRMAR - UMR 6625\\ F-35000
  Rennes, France} 
\address[K. Carrapatoso]{CMLS\\\'Ecole
  polytechnique\\ Institut Polytechnique de Paris\\ 91128 Palaiseau
  cedex\\ France}
\address[M. Hillairet]{Institut Montpelli\'erain
  Alexander Grothendieck\\ Univ. Montpellier\\ CNRS\\ Montpellier\\ France}
\email{Remi.Carles@math.cnrs.fr}
\email{kleber.carrapatoso@polytechnique.edu}
\email{Matthieu.Hillairet@umontpellier.fr}
\begin{document}

\begin{abstract}
In this paper we analyze the large-time behavior of weak solutions to polytropic fluid models possibly including quantum and viscous effects.  Formal a priori estimates show that the density of solutions to these systems should disperse with time. Scaling appropriately the system, we prove that, under a reasonable assumption on the decay of energy, the density of weak solutions converges in large times to an unknown profile. In contrast with the isothermal case,
we also show that there exists a large variety of asymptotic
profiles. We complement the study by providing existence of
global-in-time weak solutions satisfying the required decay of
energy. As a byproduct of our method, we also obtain results
concerning the large-time behavior of solutions to nonlinear
Schr\"odinger equation, allowing the presence of a semi-classical
parameter as well as long range nonlinearities.
\end{abstract}

\thanks{RC acknowledges support from Rennes M\'etropole, through its AIS program. KC acknowledges support from the Project EFI (ANR-17-CE40-0030) of the French National Research Agency (ANR). MH acknowledges support from Institut Universitaire de France and   SingFlows project (ANR-18- CE40-0027) of the French National Research Agency (ANR)}
\maketitle

\tableofcontents

\section{Introduction}
\label{sec:intro}

We consider isentropic compressible fluid models describing the evolution of the density and velocity field of a fluid on the whole space $\R^d$, given by 
\begin{align}
& \partial_t \ro + \Div(\ro u) = 0, \label{fluide1} \\
& \partial_t(\ro u) + \Div (\ro u \otimes u) + \nabla P(\ro)
= \Div \left( \frac{\eps^2}{2} \mathbb K [\ro]
+ \nu  \sqrt{\ro} \mathbb S[\ro,u] \right)   . \label{fluide2}
\end{align}
Here $t \in \R_+$ is the time variable and $x \in \R^d$ is the spatial
variable, $\ro = \ro(t,x) : \R_+ \times \R^d \to \R_+$ is the density
of the fluid, $u=u(t,x) : \R_+ \times \R^d \to \R^d$ is the velocity
field, and $P(\ro)$ is the pressure which depends on the density. The
terms appearing on the right-hand side of \eqref{fluide2} account for
quantum correction and diffusion term, respectively. More precisely, we consider constants $\eps \ge 0$ and $\nu \ge 0$ and 
\begin{equation}\label{tensorK}
\frac12 \mathbb K[\ro] = \frac14 \ro \nabla^2 \ln \ro = \frac12 (\sqrt \ro \nabla^2 \sqrt \ro - \nabla \sqrt \ro \otimes \nabla \sqrt \ro),
\end{equation}
\begin{equation}\label{tensorNS}
\mathbb S [\ro,u] = \sqrt{\ro}\D u   ,
\end{equation}
where $\D u = \tfrac12 (\nabla u + \nabla u^\top)$ denotes the
symmetric part of the gradient. We also denote below
$\A u = \tfrac12 (\nabla u -
\nabla u^\top)$ the skew-symmetric one. We emphasize that
we have split the term $\rho \D u = \sqrt{\rho}
\sqrt{\rho} \D  u$ in preparation for the following computations.  We emphasize also that we allow the
parameters $\eps$ and $\nu$ to vanish separately or
simultaneously. We will classically refer to the equations obtained according to
these cases as:
\begin{itemize}
\item Euler: $\eps=\nu=0$.
\item Euler-Korteweg: $\eps>0=\nu$.
\item Navier-Stokes (with degenerate viscosity): $\nu>0=\eps$.
\item Quantum Navier-Stokes  (with degenerate
    viscosity): $\nu>0$ and $\eps>0$. 
\end{itemize}
As for the pressure, we shall consider hereafter laws corresponding to polytropic fluids
\begin{equation}\label{pressure} 
P(\ro) = \ro^\gamma \quad\text{with}\quad \gamma >1.
\end{equation}
We refer the reader to \cite{CCH2,VasseurYu,AntonelliHientzschSpirito} for more details on the modeling.
In previous studies \cite{CCH1,CCH2}, we focused on the particular
case $\gamma=1$, which corresponds to isothermal fluids. Our main motivation herein is to show that some tools
of these former analyses extend to the polytropic case and yield
relevant information on the large-time behavior of the density $\ro$.

\subsection{Main results} System \eqref{fluide1}--\eqref{fluide2} has, at least formally, two fundamental properties: the conservation of mass
\begin{equation} \label{eq_consmass}
\int_{\R^d} \ro(t,x) \, \dx = \int_{\R^d} \ro_0(x) \, \dx \quad \forall t \ge 0,
\end{equation}
and the energy identity
\begin{equation}\label{dEdt}
E[\ro,u](t) + \nu \int_0^t D[\ro,u](s) \, \d s = E_0 \quad \forall t \ge 0,
\end{equation}
where the energy $E=E[\ro,u]$ is defined by 
\begin{equation}\label{energy}
E[\ro,u] :=  \frac12 \int_{\R^d} (\ro |u|^2 + \eps^2 |\nabla \sqrt \ro|^2 )  \, \dx + \frac{1}{\gamma-1}\int_{\R^d} \ro^\gamma \, \dx,
\end{equation}
and its dissipation by 
\begin{equation}\label{dissipation}
D[\ro,u] := \nu \int_{\R^d} \ro |\D u|^2 \, \dx.
\end{equation}

For sufficiently localized solutions  $(\ro , u)$ of \eqref{fluide1}--\eqref{fluide2} we have also at hand
alternative functionals which are well-suited for the study of large-time behavior. First, we observe that 
\[
\begin{aligned}
\frac{\d}{\dt}  \int_{\R^d} \ro |x|^2
= -\int_{\R^d} \Div(\ro u) |x|^2 =2 \int_{\R^d} \ro u \cdot x,
\end{aligned}
\]
and Cauchy-Schwarz inequality  yields
\[
\begin{aligned}
\frac{\d}{\dt}  \int_{\R^d} \ro |x|^2
\le 2\left( \int_{\R^d} \ro |x|^2\right)^{1/2} \left( \int_{\R^d} \ro |u|^2\right)^{1/2} .
\end{aligned}
\]
In view of  the above energy identity, we infer, for all $t\ge 0$,
\begin{equation}\label{rho:moment2}
 \int_{\R^d} \ro |x|^2 (t) \le C(E_0)(1+t^2),
\end{equation}
for some constant $C(E_0)>0$ depending on the energy $E_0$ of the initial data. Second, defining the functional
\[
B[\ro,u] := \frac12 \int_{\R^d} \left( \ro \left|u - \frac{x}{t}\right|^2 + \eps^2 |\nabla \sqrt \ro|^2  \right)+ \frac{1}{\gamma-1}\int_{\R^d} \ro^\gamma,
\]
we may adapt computations of  \cite[Section 4]{Serre97} (see also
\cite{Xin98}) to obtain that, for some constant $C>0$:
\begin{equation}\label{rhou:decay}
B[\ro,u](t) \le \frac{C(E_0)}{(1+t)^{\min(2,d(\gamma-1))}} + \frac{C \nu }{1+t}, \qquad \forall \, t >0.
\end{equation}
For completeness we provide an exhaustive proof of this estimate in Appendix \ref{sec:computations}. 
From \eqref{rhou:decay} we infer that the $L^{\gamma}$-norm of $\ro$ has to decay to zero with time.
Since the total mass is conserved in view of \eqref{eq_consmass}, we
expect the density of solutions to \eqref{fluide1}-\eqref{fluide2} to
disperse. To give relevant information on the asymptotic state of the
density, we consider the ($L^1$-unitary in space) rescaling
\[
R(t,x) = \tau^d(t) \ro(t,\tau(t) x) ,\quad \forall \,t >0,
\]
for some well-chosen scaling-parameter family $t \mapsto \tau(t).$ In order to compel with the 
energy bounds \eqref{rho:moment2}-\eqref{rhou:decay}, one notices that a natural choice is 
\begin{equation} \label{eq:astau}
\tau(t) \sim t.
\end{equation}
 The precise choice of $\tau$ will be motivated by the fact that our approach relies on compactness properties requiring a priori estimates of the form \eqref{rhou:decay}. Typically, in the case
$\nu=0$, we use different choices whether $\gamma$ is smaller or
larger than $1+2/d$. 

\medskip

In the case of isothermal models $\gamma=1$, we showed that, by choosing the
scaling $\tau$ as solution to an appropriate differential equation
(which can be inferred {\em via} a parallel with nonlinear Schr\"odinger
equations), we may not only analyze the asymptotics of solutions to
\eqref{fluide1}-\eqref{fluide2} but also identify new energy estimates
crucial to the construction of a Cauchy theory for this system (see
\cite{CCH1,CCH2}). Such ideas were already hinted in
\cite{Serre97}. 
Following these previous approaches, we propose to introduce the
scaling function $\tau$ as follows. We fix $\alpha >0$ and compute
$\tau : \R_+ \to \R_+$ as the solution of
 \begin{equation}\label{eq:tau}
     \ddot \tau = \frac{\alpha}{2\tau^{1+\alpha}},\quad \tau(0)=1,\ \dot
  \tau(0)=0.
\end{equation}
Its large time behavior turns out to be independent of $\alpha>0$:
\begin{lemma}\label{lem:tau}
  Let $\alpha>0$. 
  The ordinary differential equation \eqref{eq:tau}
 has a unique, global, smooth solution $\tau\in C^\infty(\R;\R_+)$. In
addition, its large time behavior is given by
\begin{equation*}
  \dot \tau(t)\Tend t {\infty} 1,\quad\text{hence }\tau(t)\Eq t
  {\infty} t.
\end{equation*}
\end{lemma}
The proof  is postponed to Appendix~\ref{sec:tau}. Define  the new unknowns $(R,U)$ by
\begin{equation}
\label{eq:uvFluid}
\ro(t,x) = \frac{1}{\tau(t)^d}R\(t,\frac{x}{\tau(t)}\),\quad 
u(t,x) = \frac{1}{\tau(t)} U \(t,\frac{x}{\tau(t)}\) +\frac{\dot \tau(t)}{\tau(t)}x.
\end{equation}
This change of unknowns does not affect the initial data:
\begin{equation*}
  R(0,x) = \ro(0,x),\quad U(0,x)= u(0,x).
\end{equation*}
In terms of $(R,U) = (R(t,y), U(t,y))$, system \eqref{fluide1}--\eqref{fluide2} then becomes, 
\begin{align}
& \partial_t R+\frac{1}{\tau^2}\Div (R U )=0    \label{RU1}  \\
& \partial_t (R U) +\frac{1}{\tau^2}\Div ( R U \otimes U) 
+\frac{\alpha}{2 \tau^\alpha} y R  + \frac{1}{\tau^{d(\gamma-1)}}\nabla R^\gamma    \label{RU2} \\ 
& \phantom{\partial_t (R U) +\frac{1}{\tau^2}\Div ( R U \otimes U) 
+}=\frac{1}{\tau^2} \Div \left( \frac{\eps^2}{2} \mathbb K[R] 
+ \nu \sqrt{R} \mathbb S[R,U] \right) 
 + \frac{\nu \dot \tau}{\tau} \nabla R, \notag  
\end{align} 
with $\mathbb K$ and $\mathbb S$ defined as previously. 
Since $\tau$ depends on some parameter $\alpha$, so do
the new unknowns $R$ and $U$, and the system \eqref{RU1}-\eqref{RU2}. However, since the large-time behavior of $\tau$ does not depend on $\alpha>0$ (see
Lemma~\ref{lem:tau}), the different changes of unknowns describe the same quantities asymptotically. The precise value of $\alpha$ is motivated by
the a priori estimates we obtain for $(R,U)$. Yet, the analysis could
be performed for various $\alpha$'s, and since the large-time behavior
of $(R,U)$ should be independent of $\alpha,$ one expects that
$\alpha$-dependent terms in \eqref{RU1}-\eqref{RU2} will be  
subdominant in large times. 

\smallbreak

The analogues of \eqref{eq_consmass}, \eqref{energy} and
\eqref{dissipation} in terms of the new unknown $(R,U)$ are the
following. First, for fixed $t$, the map $\ro(t,\cdot)\mapsto
R(t,\cdot)$ preserves the $L^1$-norm, hence \eqref{eq_consmass}
becomes
\begin{equation}\label{massR}
  \int_{\R^d}R(t,y)\dy = \int_{\R^d}\ro_0(x)\dx.
\end{equation}
Define the \emph{pseudo-energy} $\EE[R,U]$ by
\begin{equation}\label{pseudo-energy}
  \begin{aligned}
    \EE[R,U] &= \frac{1}{2 \tau^2} \int_{\R^d} \left( R|U|^2 + \eps^2 |\nabla \sqrt R|^2 \right) \dy 
    + \frac{\alpha}{4\tau^\alpha} \int_{\R^d} |y|^2 R \, \dy \\
    &\quad
+\frac{1}{(\gamma-1) \tau^{d(\gamma-1)}} \int_{\R^d} R^\gamma  \, \dy,
 \end{aligned}
\end{equation}
as well as its nonnegative dissipation $\DD$ by
\begin{equation}\label{dissipationRU}
  \begin{aligned}
    \DD[R,U] &= \frac{\dot \tau}{ \tau} \Bigg\{  \frac{1}{\tau^2} \int_{\R^d} \left( R|U|^2 + \eps^2 |\nabla \sqrt R|^2 \right)  \dy + \frac{\alpha^2}{4 \tau^\alpha}\int_{\R^d} |y|^2 R \, \dy
    + \frac{d}{\tau^{d(\gamma-1)}} \int_{\R^d} R^\gamma \, \dy \Bigg\} \\
&\quad + \frac{\nu}{\tau^4} \int_{\R^d} R | \D U|^2.
  \end{aligned}
\end{equation}
We obtain, at least formally, the following pseudo-energy identity
\begin{equation}\label{dEEdt}
\frac{\d}{\dt} \EE[R,U] + \DD[R,U]  
= - \nu \frac{\dot \tau}{ \tau^3} \int_{\R^d} R \Div U \, \dy .
\end{equation}
We remark that $\EE$ does not correspond to the energy $E$ written in
the $(R,U)$ variables. These formal identities imply that densities
$(R(t,\cdot))_{t>0}$ are positive with finite mass and second
moment. An appropriate functional space to tackle the large-time
behavior is then the set of positive measures on $\mathbb R^d$.  Up to
a scaling argument -- which may only change the amplitude  
of pressure law and Korteweg terms --  we restrict to the case where
$(R(t,\cdot))_{t>0}$ is a family of probability measures.

\medskip


\begin{notations}
  We use classical notations ${C}^{\infty}_c(\mathbb R^d),$ $\mathcal S(\mathbb R^d)$ for smooth functions with compact support and Schwartz space. Notations $L^{p}(\mathbb R^d)$ (resp. $H^{s}(\mathbb R^d), W^{m,p}(\mathbb R^d)$) refer to Lebesgue (resp. Sobolev spaces). We shall make repeated use of Bochner spaces 
$L^p(0,\infty;L^{q}(\mathbb R^d))$, of  $L^{p}(0,\infty;H^s(\mathbb R^d))$, and their local-in-time variants.
In the space $L^{\infty}_{\rm loc}(0,\infty;L^p (\mathbb R^d))$ we denote $C([0,\infty);L^{p}(\mathbb R^d)-w)$ 
the subspace of continuous functions when endowing $L^{p}(\mathbb R^d)$ with its weak topology.
The space $\mathcal D'(\Omega)$ is made of distributions on the open set $\Omega$ (not to be confused with $\mathcal D$).  We denote by 
$ \mathbb P(\mathbb R^d)$ the set of probability measures on $\R^d$. More generally, for $j\in \N$,
$\mathbb P_j(\mathbb R^d)$ denotes the space of probability measures on
$\R^d$ with finite $j$-th moment. 
\end{notations}

\subsubsection{Rigidity result}
Our main contribution consists in analyzing large-time properties of 
potential weak solutions to \eqref{RU1}-\eqref{RU2}. 
Building up on our previous
construction in \cite{CCH2}  we consider  weak solutions
which read $(\sqrt{R},\sqrt{R}U)$ and that enjoy  the following properties:
\begin{itemize}
\item[(H1)] $\sqrt{R} \in L^{\infty}(0,\infty;L^2(\mathbb R^d)) \cap
  L^{\infty}_{\rm loc}(0,\infty;L^{2\gamma}(\mathbb R^d))$, $ (\eps+\nu) 
  \sqrt{R} \in L^{\infty}_{\rm loc}(0,\infty;H^1(\mathbb R^d)),$ 
  with $R(t,\cdot) \in \mathbb P_2(\mathbb R^d)$ for a.e.\ $t >0,$
\item[(H2)] $\sqrt{R} U \in L^{\infty}_{\rm loc}(0,\infty;L^{2}(\mathbb R^d)),$
\item[(H3)] There exists  $\mathbb{T} \in L^2_{\rm loc}(0,\infty;L^2(\mathbb R^d)),$ such that
\begin{equation*}
\left\{
\begin{aligned}
& \partial_t R + \dfrac{1}{\tau^2}{\rm div}(\sqrt{R} \sqrt{R} U) = 0 \\
& \partial_t(\sqrt{R} \sqrt{R}U) + \dfrac{1}{\tau^2}{\rm div}(\sqrt{R}U \otimes \sqrt{R}U) + \dfrac{\alpha}{2\tau^{\alpha}} yR +  \dfrac{1}{\tau^{d(\gamma - 1)}} \nabla R^{\gamma} 
\\
& \phantom{\partial_t(\sqrt{R} \sqrt{R}U) + \dfrac{1}{\tau^2}{\rm div}(\sqrt{R}U \otimes \sqrt{R}U) }= \dfrac{1}{\tau^2}{\rm div} \left( \frac{\eps^2}{2} \mathbb K + \nu\sqrt{R}\mathbb T^s  \right) + \dfrac{\nu \dot{\tau}}{\tau} \nabla R,
\end{aligned}
\right.
\end{equation*}
holds in $\mathcal D'((0,\infty)\times \mathbb R^d)$
with the compatibility conditions,
when these terms are present:
\[
\mathbb K = \sqrt R \nabla^2 \sqrt R - \nabla \sqrt R \otimes \nabla \sqrt R, \qquad 
\sqrt R \mathbb T =\nabla (\sqrt{R} \sqrt{R}U) - 2 \sqrt{R}U
\otimes \nabla \sqrt{R}.
\]
\end{itemize} 
For legibility, we have written equations in terms of 
$R=\(\sqrt R\)^2$ 
in this definition whereas, since $\sqrt{R}$ is the involved unknown,
these quantities must be computed in terms of $\sqrt{R}.$ Similarly,
$U$ is not an appropriate unknown in our framework. So, we do not
write the quantity $\sqrt{R} \nabla U$ but the symbol $\mathbb T$
which plays its role 
(expressed in terms of $(\sqrt R,\sqrt R U)$),
hence our second compatibility condition in (H3).
The exponent $s$ denotes here the symmetric part of $\mathbb T.$
Such assumptions are also inspired by the definition of weak solution
in \cite[Definition~1.1]{CCH2} (isothermal case), as in \cite[Definition 2.1]{AntonelliHientzschSpirito}), with further momenta requirements for the density (see also \cite{LacroixVasseur} in the case of the torus).
\medskip

To complete the set of assumptions, it is mandatory to enforce in one
way or another the decay properties inherited from \eqref{dEEdt} (as
it is classical for weak solutions to dissipative systems). By abuse of notations, we keep the symbols $\EE$ for energy  and $\mathcal D$ for its dissipation, though they will be computed in terms of $\sqrt{R}, \sqrt{R}U$ and $\mathbb T$ 
(and not $R$ and $U$ which are not the good unknowns in this weak-solution framework). 
Our last requirement builds on the following formal analysis.
First we bound the right-hand side in \eqref{dEEdt} as follows: 
\[
\begin{aligned}
\frac{\d}{\dt} \EE  + \DD
&\le \nu \frac{\dot \tau}{ \tau^3} \left( \int_{\R^d} R \, \dy \right)^{1/2} \left( \int_{\R^d} \left|\mathbb T^s\right|^2 \, \dy\right)^{1/2} \\
&\le 2\nu \frac{(\dot \tau)^2}{\tau^2} \int_{\R^d} R \, \dy + \frac{\nu}{2 \tau^4} \int_{\R^d} |\mathbb T^s|^2 \, \dy,
\end{aligned}
\]
which implies 
\[
\frac{\d}{\dt} \EE  + \frac12 \DD
\le 2\nu \frac{(\dot \tau)^2}{\tau^2} .
\]
Remarking that $\int_0^\infty \frac{(\dot \tau(t))^2}{\tau(t)^2} \d t<
\infty$ (see Lemma~\ref{lem:tau}), we already obtain that 
\begin{equation} \label{eq_dEEt0}
\sup_{t \ge 0} \EE(t) + \int_0^\infty \DD(t) \, \dt \le C(\EE_0),
\end{equation}
where $C(\EE_0)>0$ is a constant depending on the pseudo-energy of the initial data $\EE_0$. This yields at first that $\DD$ is in $L^{1}(0,\infty).$
Observing from \eqref{dissipationRU} that
if $\alpha\le\min(2,(d(\gamma-1))$,
\[
\DD[R,U] \ge \alpha \frac{\dot \tau}{\tau} \EE[R,U],
\]
we therefore deduce the following differential inequality
\begin{equation} \label{eq_dEEt2}
\frac{\d}{\dt} \EE
\le - \frac{\dot \tau}{\tau} \alpha \EE +   2\nu \frac{(\dot \tau)^2}{\tau^2} ,
\end{equation}
which entails after integration that, for all $t \ge 0$ (the outcome is slightly different whether
$\alpha\not =1$ or $\alpha=1$),
\begin{equation}\label{RU:decay}
\EE(t) \le C_0\left( \frac{1}{1+t^\alpha} + \frac{\nu}{1+t} (\mathbf{1}_{\alpha \neq 1}+\ln (1+t) \, \mathbf{1}_{\alpha=1} )\right).
\end{equation}
We note that, when $\nu >0$ and $\alpha > 1,$ these computations only
yield a bound on the growth of the second 
moment of $R(t).$ This can be improved thanks to the
following remark. By multiplying the
continuity
equation with $|y|^2$ we have formally that:
\[
\frac{\d}{\dt} \int_{\R^d} |y|^2 R \, \dy  = \frac{2}{\tau^2} \int_{\R^d} R y \cdot U \, \dy,
\]
which implies 
\[
\begin{aligned}
\frac{\d}{\dt} \int_{\R^d} |y|^2 R \, \dy  
\le  \frac{2}{\tau^2} \left( \int_{\R^d} R |y|^2 \, \dy \right)^{\frac12} \left( \int_{\R^d} R |U|^2 \, \dy \right)^{\frac12} ,
\end{aligned}
\]
and thus
\begin{equation}\label{Ry2}
\begin{aligned}
\frac{\d}{\dt} \left(\int_{\R^d} |y|^2 R \, \dy \right)^{\frac12} 
\le  \frac{1}{\tau} \left( \frac{1}{\tau^2} \int_{\R^d} R |U|^2 \, \dy \right)^{\frac12} \le \frac{C}{\tau} \sqrt{\EE}.
\end{aligned}
\end{equation}
The combined decay of $\EE$ and growth of $\tau$ entail finally that
the second 
moment of $R$ remains bounded whatever the value of $\alpha.$ 

\medskip

Eventually, these formal considerations lead us to the following last assumption:
\begin{itemize}
\item[(H4)] Set $\alpha=\min(2,d(\gamma-1))$.  Introducing $\mathcal E,\mathcal D$ as defined previously (see \eqref{pseudo-energy} and \eqref{dissipationRU}), there exists a constant $C_0 >0$ such that:
\begin{align} \label{eq_diss1}
 &\mathcal E(t) \le C_0 \left( \dfrac{1}{(1+t)^{\alpha}} + \dfrac{\nu}{(1+t)} (\mathbf{1}_{\alpha \neq 1}+\ln (1+t) \, \mathbf{1}_{\alpha=1} )\right) , \qquad \forall \, t >0, \\
 & \sup_{t >0} \left( \int_{\mathbb R^d} |y|^2 R(t,y){\rm d}y \right) +  \int_0^{\infty} \mathcal D(t) \, \dt  \le C_0.
\label{eq_diss2}
\end{align}
\end{itemize}
More details on the derivation of \eqref{eq_diss1}-\eqref{eq_diss2} are given in Section~\ref{sec:existence}.
With these assumptions, our main result yields a description of the large-time behavior of the density $R(t) = [\sqrt{R}(t)]^2.$ This is the content of the following theorem:
\begin{theorem} \label{theo:rigidity}
Assume that $(\sqrt{R},\sqrt{R}U)$ is a global weak solution to \eqref{RU1}--\eqref{RU2} such 
that (H1)--(H4) hold true. There exists $R_{\infty} \in \mathbb P_2(\mathbb R^d)$
such that
\[
R(t,\cdot) \rightharpoonup R_{\infty} \quad \text{ in }  \mathbb P(\mathbb R^d).
\]
We have in addition $R_\infty\in L^1(\R^d)$ (at least) in the
following cases:
\begin{itemize}
\item $\eps = \nu=0$ and $1<\gamma\le 1+2/d$,
\item $\eps>0$, $\nu=0$ and $ \gamma >1 $,
\item $\eps \ge 0$, $\nu>0$ and $1<\gamma\le 1+1/d$.
\end{itemize}
\end{theorem}

\begin{remark}
In \cite{AnMaZh21}, the authors consider the Korteweg case $\eps>0=\nu$ in
dimension
$d=1$, and show dispersive estimates of another nature, in the sense that they involve a strong topology: under (H1), (H2) and (H3) (the assumptions actually involve $(\ro,u)$ instead of $(R,U)$, which means in particular that one should consider $\tau\equiv 1$ in (H3)), the authors prove that $\ro\in L^{\gamma+1}(\R_t\times \R^d_x)$ and $\partial_x \ro \in L^2(\R_t\times \R^d_x)$. The proof relies on an adaptation of Morawetz estimates, which are a classical tool in the study of nonlinear Schr\"odinger equations; see e.g. \cite{GiVe10} for a presentation which clearly uses the link between the nonlinear Schr\"odinger equation and hydrodynamical equations. 
\end{remark}

We obtain in the course of the proof an explicit polynomial rate of convergence from $R(t,\cdot)$ to $R_{\infty}.$ 
Reconstructing the solution $(\ro,u)$ from $(R,U)$ via the formulas
\eqref{eq:uvFluid}, we infer
\begin{equation*}
  \lim_{t \to \infty} \tau^d(t) \ro(t,\tau(t) \cdot)  = R_{\infty} \text{ in $\mathbb P(\mathbb R^d)$}.
\end{equation*}

\medskip

We emphasize that contrary to the isothermal case $\gamma=1$, where, as proven in \cite{CCH1}, 
the only possible $R_\infty$ is given by
\begin{equation*}
  R_\infty(y) =\frac{\|\ro_0\|_{L^1(\R^d)}}{\pi^{d/2}} e^{-|y|^2},
\end{equation*}
in the polytropic case $\gamma>1$, the range of the map $\rho_0\mapsto
R_\infty$ is very broad. In the case of the Euler equation, we have,
as established in \cite{CCH1} by adapting the approach from
\cite{Serre97}:
\begin{proposition}
  Let  $\eps=\nu=0$, $1<\gamma\le 1+2/d$ and $s>d/2+1$. There exists
  $\eta>0$ such 
  that if $0\le a_\infty\in H^s(\R^d)$ is such
  that 
  $\|a_\infty\|_{H^s(\R^d)}\le \eta$, then
there exists a solution to
\eqref{fluide1}-\eqref{fluide2}  which is global in time,  with
 \begin{equation*}
    \left\| \ro
      (t,x)-\frac{1}{t^d}R_\infty\(\frac{x}{t}\)\right\|_{L^\infty(\R^d)\cap
    L^1(\R^d)}\Tend
    t \infty 0, \quad R_\infty :=a_\infty^{\frac{2}{\gamma-1}}.
  \end{equation*}
\end{proposition}
In the case of the Euler-Korteweg system, we will prove (for a different
range of $\gamma$): 
\begin{proposition}\label{prop:wave-op-korteweg}
  Let $\eps>0=\nu$,
  \begin{equation*}
    \gamma>3 \text{ if }d=1,\quad
    1+\frac{4}{d+2}<\gamma<1+\frac{4}{(d-2)_+} \text{ if }d\ge 2.
  \end{equation*}
  For any $a_\infty\in \mathcal S(\R^d)$, there exists a solution to
  \eqref{fluide1}-\eqref{fluide2} such that
  \begin{equation*}
    \left\| \ro(t,x)-\frac{1}{t^d}R_\infty\(\frac{x}{t}\)\right\|_{L^1(\R^d)}\Tend
    t \infty 0, \quad R_\infty :=|a_\infty|^2.
  \end{equation*}
\end{proposition}
This result is a direct consequence of scattering theory for nonlinear
Schr\"odinger equations, as discussed more precisely in
Section~\ref{sec:wave-op-korteweg}.

\subsubsection{Existence results}
The second natural contribution consists in making sure that the assumptions of
Theorem~\ref{theo:rigidity} are not empty. 

\begin{theorem}\label{theo:existence}
In the three following cases, initial data $(\rho_0,u_0)$ yield at least one global weak solution $(\sqrt{R},\sqrt{R}U)$ to \eqref{RU1}-\eqref{RU2} satisfying the assumptions of Theorem~\ref{theo:rigidity}:\\[-6pt]
\paragraph{(i) {\bf Euler equations}.} Assume $\eps=\nu=0$. Let $\gamma>1$, $s > d/2 +1$
and $r_0\in H^s(\R^d)$  such that $r_0\ge 0$ is compactly supported
with $\left\|r_0\right\|_{H^s(\R^d)}$ sufficiently small. Then, assume
$\ro_0(x) =r_0(x)^{\frac{2}{\gamma-1}}$, and  $u_0$  satisfies
$\nabla^{2}u_0\in H^{s-1}(\R^d)$, $\nabla u_0\in
  L^\infty(\R^d)$, and there exists $\delta>0$ such that for all
  $x\in \R^d$, ${\rm dist}(\operatorname{Sp}(\nabla u_0(x)),\R_-)\ge \delta$, 
where we denote by $\operatorname{Sp}(M)$ the
spectrum of a matrix $M$. \\[-6pt]
\paragraph{(ii) {\bf Euler-Korteweg equations}}
Assume $\eps >0$, $\nu = 0$ and $1<\gamma<1+\frac{4}{(d-2)_+}$, and there exists
   \begin{equation*}
     \psi_0\in \Sigma:= \{f\in H^1(\R^d),\quad x\mapsto xf(x)\in
     L^2(\R^d)\},
   \end{equation*}
   such that $\ro_0 = |\psi_0|^2$, $\ro_0u_0 = \eps \IM(\bar
   \psi_0\nabla \psi_0)$.\\[-6pt]
\paragraph{(iii) {\bf Navier-Stokes equations}.}  
Assume $d\le 3,$ $\gamma >1,$ $\nu >0$ and $\eps \geq 0.$ Let $(\ro_0,u_0)$ satisfy:
  \begin{equation*}
(1+|x|+|u_0|)\sqrt{\ro_0} \in L^2(\mathbb R^d), \quad \ro_0 \in L^{\gamma}(\mathbb R^d), \quad \sqrt{\ro_0} \in H^1(\mathbb R^d) .
\end{equation*}
\end{theorem}

We remind that the change of unknown from small-letter to capital-letter unknowns does not affect initial data. In particular, depending on the case, we may prefer to solve the small-letter system \eqref{fluide1}-\eqref{fluide2} and then apply the change of unknown to yield weak solutions satisfying (H1)-(H4) or directly work on the scaled system \eqref{RU1}-\eqref{RU2} with the capital-letter unknowns. More details are given in Section \ref{sec:existence}. 

\subsection{Nonlinear Schr\"odinger equation}
\label{sec:NLS0} 
It is well-known (see e.g. \cite{AnMa09,CaDaSa12}) that the
Euler-Korteweg equation is intimately related to 
the nonlinear Schr\"odinger equation (NLS)
\begin{equation}
  \label{eq:nls}
  i\eps\partial_t \psi^\eps +\frac{\eps^2}{2}\Delta \psi^\eps = \lambda
  |\psi^\eps|^{2\si}\psi^\eps, \quad \psi^\eps_{\mid t=0}=
  \psi_0^\eps\in H^1(\R^d),
\end{equation}
through the Madelung transform,
\begin{equation}
  \label{eq:madelung}
  \ro = |\psi^\eps|^2,\quad \ro  u =\eps\IM\(\bar \psi^\eps\nabla
  \psi^\eps\),\quad \text{with}\quad
  \lambda = \frac{\gamma}{\gamma-1},\quad \si =\frac{\gamma-1}{2} .
\end{equation}
We emphasize the dependence of $\psi^\eps$ upon $\eps$ through the
notation, for the limit $\eps\to 0$ corresponds to the semi-classical
limit, and will be discussed in the present paper. The Cauchy problem
\eqref{eq:nls} is easier than its fluid mechanical counterpart: if
$\lambda>0$ and  
$0<\si<\tfrac{2}{(d-2)_+}$  (defocusing, energy-subcritical
  nonlinearity), then  \eqref{eq:nls}  has a unique solution
  \begin{equation*}
    \psi^\eps \in C(\R;H^1(\R^d))\cap L^{\frac{4\si+4}{d\si}}_{\rm
      loc}(\R;L^{2\si+2}(\R^d)). 
  \end{equation*}
 See e.g.~\cite{CazCourant}. If in addition $x\mapsto x\psi_0^\eps\in
 L^2(\R^d)$, then this  integrability property is propagated by the
 flow. The 
analogue 
in the context of nonlinear Schr\"odinger equations
of the evolution of $B[\ro,u]$ 
in the fluid mechanical case
was discovered by Ginibre and Velo~\cite{GV79scatt} 
(and thus actually before its counterpart in fluid mechanics),
and goes under the name of
 \emph{pseudo-conformal conservation law}. 
 \begin{theorem}\label{theo:NLS}
   Let $d\ge 1$, $\eps,\lambda>0$, $0<\si<\frac{2}{(d-2)_+}$, and
   \begin{equation*}
     \psi_0^\eps\in \Sigma:= \{f\in H^1(\R^d),\quad x\mapsto xf(x)\in
     L^2(\R^d)\}. 
   \end{equation*}
   Rescale the function $\psi^\eps$ to $\Psi^\eps$ via
   \begin{equation}\label{eq:psiPsi}
  \psi^\eps(t,x) =
  \frac{1}{\tau(t)^{d/2}}\Psi^\eps\(t,\frac{x}{\tau(t)}\)e^{i\frac{\dot
      \tau(t)}{\tau(t)}\frac{|x|^2}{2\eps}} \|\psi_0^\eps\|_{L^2(\R^d)},
\end{equation}
where $\tau(t)$ is a scaling like before (in particular, $\tau(t)\sim
t$ as $t\to \infty$). There exists $R^\eps_\infty\in \mathbb P_2(\R^d)
$ such that
\begin{equation*}
  |\Psi^\eps(t,\cdot)|^2 \rightharpoonup R^\eps_{\infty} \quad \text{ in }
  \mathbb P(\mathbb R^d). 
\end{equation*}
 \end{theorem}
More details are given in
Section~\ref{sec:nls}. At this stage, we emphasize the fact that $\si$
is arbitrarily small. In particular, for $0<\si\le 1/d$, the
nonlinearity is long range, in the sense that no standard scattering
result is possible: fix $\eps>0$, and assume that there exists
$\psi_+^\eps\in L^2(\R^d)$ such that
\begin{equation}\label{eq:scattering}
  \left\|\psi^\eps(t) - e^{i\eps
      \frac{t}{2}\Delta}\psi^\eps_+\right\|_{L^2(\R^d)}\Tend t \infty 0,
\end{equation}
then necessarily $\psi^\eps\equiv 0$, from \cite{Barab}. On the other
hand, it is a common belief that long range effects affect only the
behavior of the phase, at leading order, meaning that the dispersion
is the same as in the linear case. Indeed, for $\si>1/d$, under the
assumptions of Theorem~\ref{theo:NLS}, there exists $\psi^\eps_+$
(with in particular $\|\psi^\eps_+\|_{L^2(\R^d)}=\|
\psi^\eps_0\|_{L^2(\R^d)}$) such that \eqref{eq:scattering} holds (\cite{TsYa84}).
Recall that in $L^2(\R^d)$ (see e.g. \cite{Tsutsumi85}),
\begin{equation*}
    e^{i\eps \frac{t}{2}\Delta}f(x)\Eq t { \infty} \frac{1}{(\eps t)^{d/2}}\hat
   f\(\frac{x}{\eps t}\)e^{i\frac{|x|^2}{2\eps t}}. 
\end{equation*}
 Therefore, for $\eps>0$ fixed, Theorem~\ref{theo:NLS} shows that
 long range effects do not alter the standard dispersion.

\subsection{Outline of the paper}
In brief, the paper splits into 3 sections and 2 appendices. 
In {\bf Section~\ref{sec:universal}} we provide a proof of {\bf Theorem~\ref{theo:rigidity}}. 
The next section is devoted to the analysis of nonlinear Schrödinger equations to provide
the examples of {\bf Proposition~\ref{prop:wave-op-korteweg}}. We complement the analysis
in {\bf Section~\ref{sec:existence}} with the proof of the existence
result {\bf Theorem~\ref{theo:existence}}. 
The two appendices are
devoted  to the formal computation of decay  
estimate \eqref{rhou:decay}, and to the properties of the scaling
parameter families $(\tau(t))_{t>0}$, respectively.


\section{Proof of Theorem \ref{theo:rigidity}}
\label{sec:universal}

We consider  non-negative parameters $\eps,\nu$, and assume that $(\sqrt{R},\sqrt{R}U)$
is a global weak solution to \eqref{RU1}-\eqref{RU2}
in the sense of (H1)-(H3), enjoying the decay properties  
(H4).

\medskip

As a preliminary, we note from (H4) that
\[
\left( \int_{\mathbb R^d} |y|^2 R(t,y)\dy \right)_{t >0} \text{ is bounded. }
\]
So the family of probability densities $(R(t,\cdot))_{t >0}$ is tight
and precompact in $\mathbb P(\mathbb R^d).$  
Remark that this already implies that there is some sequence of times
$(t_n)_{n \ge 0}$ with $t_n \to \infty$ as $n \to \infty$, such that
$(R(t_n,\cdot))_{n \ge 0}$ converges weakly in $\mathbb P(\mathbb
R^d)$ to some probability measure $R_\infty$. Unlike in the isothermal
case, we have not been able to identify a limiting equation for
$R_\infty$, which could make it possible to infer uniqueness of the
accumulation point ($R(t,\cdot)$ might keep oscillating as $t\to
\infty$). 
However, given the uniform bound on $(R(t,\cdot))$ in $\mathbb P_2(\mathbb R^d),$ our proof reduces to obtaining convergence in some sufficiently large dual space.  To this end, we will make repeated use of the following lemma:
\begin{lemma} \label{lem_reg}
Let $T >0,$ $m \in \mathbb N$ and $(p,q) \in (1,\infty).$ Assume that $X \in L^{\infty}(0,T ; L^p(\mathbb R^d))$
satisfies $\partial_t X \in L^1(0,T;W^{-m,q}(\mathbb R^d)).$ Then there holds:
\begin{itemize}
\item $X \in C([0,T] ; L^p(\mathbb R^d)-w)$
\item for arbitrary $\varphi \in C^{\infty}_c(\mathbb R^d)$ there holds:
\[
\left[ \int_{\mathbb R^d} X(\cdot,y) \varphi(y){\rm d}y \right]_{t_1}^{t_2}= \left\langle \partial_t X ,(t,y) \mapsto \varphi(y) \mathbf{1}_{[t_1,t_2]}(t) \right\rangle  \quad \forall \, 0 \leq t_1 < t_2 \leq T. 
\]
\end{itemize}
\end{lemma}
This lemma  is part of the folklore and is stated without proof. Formally, it is tempting to invoke \eqref{RU1} and use Cauchy-Schwarz
inequality to obtain
\begin{equation*}
  \int_{\R^d}R|U|\dy\le
  \(\int_{\R^d}R\dy\)^{1/2}\(\int_{\R^d}R|U|^2\dy\)^{1/2}\lesssim \tau.
\end{equation*}
Then one may want to write, in view of \eqref{RU1},
\begin{equation*}
  \|\partial_t  R\|_{W^{-1,1}(\R^d)}=\frac{1}{\tau^2}\|\Div(RU)\|_{W^{-1,1}(\R^d)}\le
  \frac{1}{\tau^2}\|RU\|_{L^1(\R^d)} \lesssim \frac{1}{\tau}.
\end{equation*}
We see that we barely miss integrability on the right hand
side, due to a logarithmic divergence. Also, this estimate implicitly
relies on duality properties of $W^{-1,1}$, which is a delicate
matter. To overcome these issues, we estimate $\partial_tR$ at a lower
regularity level in order to obtain integrability in time, and
we consider estimates related to $L^p$ spaces with $1<p<\infty$ (for
reflexivity), and $p>d$ so we can use
Sobolev embeddings $W^{s,p}(\R^d)\subset W^{s-1,\infty}(\R^d)$. This
again reduces the level of regularity at which we estimate
$\partial_tR$. More precisely, we estimate
$\|\partial_tR\|_{W^{-4,p'}}$ for $d<p<\infty$, in
{\bf Proposition~\ref{prop:sob}} below.
\smallbreak

The core of the proof is then two successive applications of {\bf Lemma~\ref{lem_reg}}.
First, we obtain:
\begin{proposition} \label{prop_RU}
Let $\gamma_* = 2\gamma/(\gamma+1).$ There holds $RU = \sqrt{R} \sqrt{R}U \in C([0,\infty);L^{\gamma_*}(\mathbb R^d)-w)$ and,
given $p > d,$ there exists $K_p >0$ depending  on $C_0$ in (H4) and $\alpha,p,\eps,\nu$ for which:
\[ 
\left| \int_{\mathbb R^d} RU \cdot w \right| \leq K_p \left( (1+ t)^{(1-\alpha)_+} + \ln(1+t)\mathbf{1}_{\alpha =1} + \mathbf{1}_{\nu >0}(1+t)^{1/2}  \right) \|w\|_{W^{3,p}(\mathbb R^d)} ,
\]
for all $w \in [C^{\infty}_c(\mathbb R^d)]^d.$
\end{proposition}
\begin{proof}
Since we have $RU \in L^{\infty}_{\rm loc}(0,\infty); L^{\gamma_*}(\mathbb R^d))$, a direct application of {\bf Lemma \ref{lem_reg}} yields our result if we prove, for any $t >0$, that:
\[
\left| \langle \partial_t (RU) , w \rangle  \right|  \leq K_p  \left( (1+ t)^{(1-\alpha)_+} + \ln(1+t)\mathbf{1}_{\alpha =1} + \mathbf{1}_{\nu >0}(1+t)^{1/2}  \right)
\sup_{s \in (0,t)} \|w(s,\cdot)\|_{W^{3,p}(\mathbb R^d)} 
\]
for arbitrary $w  \in C^{\infty}_c((0,t) \times \mathbb R^d)^d$.  
To this respect, we will make repeated use without mention of the
property, stemming from \eqref{RU:decay}:
\begin{align*}
\int_{0}^t \mathcal E(s) \, {\rm d}s & \leq C_0 \int_0^{t} \left[ \dfrac{1}{(1+s)^{\alpha}} + \dfrac{\nu}{(1+s)}\left(\mathbf{1}_{\alpha \neq 1} + \ln(1+s) \mathbf{1}_{\alpha=1} \right)\right]{\rm d}s \\
& \leq K_p \left( (1+t)^{(1-\alpha)_+} + \ln(1+t)\mathbf{1}_{\alpha =1}  + \mathbf{1}_{\nu >0} \sqrt{1+t}\right).
\end{align*}
So, given $t>0$ and $w \in C^{\infty}_c((0,t) \times \mathbb R^d)^d,$ we apply \eqref{RU2}
to split: 
\[
\langle \partial_t (RU), w \rangle = \sum_{k=1}^{5} \langle L_i ,w \rangle ,
\]
where:
\begin{align*}
\langle L_1 , w \rangle  &= - \int_{0}^t\dfrac{1}{\tau^2} \int_{\mathbb R^d} (\dfrac{\eps^2}{2} \mathbb K + \nu \sqrt{R}\mathbb S) : \nabla w,  \\
\langle L_2 , w \rangle  &= - \int_0^t \dfrac{\nu \dot{\tau}}{\tau} \int_{\mathbb R^d} R \Div w ,\\
\langle L_3 , w \rangle  &= \int_{0}^t \dfrac{1}{\tau^2} \int_{\mathbb R^d} \sqrt{R}U \otimes \sqrt{R}U : \nabla w ,\\
\langle L_4 , w \rangle &= - \int_{0}^t \frac{\alpha}{2 \tau^\alpha} \int_{\mathbb R^d}  R y \cdot w ,\\
\langle L_5, w \rangle & = \int_{0}^t \dfrac{1}{\tau^{d(\gamma-1)}} \int_{\mathbb R^d} R^{\gamma} \cdot \Div w.    
\end{align*}
We now estimate  these five terms independently.

\medskip

Concerning $L_1,$ we split $L_1 =  L_{1}[\mathbb K] +  L_1[\mathbb S]$ with obvious notations. First we bound:
\begin{align*}
| \langle L_1[\mathbb K] , w \rangle |  & = 
\left | \int_{0}^t\dfrac{\eps^2 }{2\tau^2}\int_{\mathbb R^d} 
\left( \sqrt{R} \nabla^2 \sqrt{R} - \nabla \sqrt{R} \otimes \nabla \sqrt{R}\right)  :  
\nabla w 
\right| \\
& \lesssim \int_0^t \dfrac{1}{\tau^2} 
\int_{\mathbb R^d}  \left( |\eps \nabla \sqrt{R}|^2 |\nabla w| + \eps \sqrt{R}  |\eps \nabla \sqrt{R}| |\nabla^2 w| \right)\\
& \lesssim  \int_0^t \left( \dfrac{\eps^2}{\tau^2} \int_{\mathbb R^d}  |\nabla \sqrt{R}|^2   + \dfrac{\eps^2}{\tau^2} \int_{\mathbb R^d} R\right)
\sup_{[0,t]} \|w\|_{W^{2,\infty}(\mathbb R^d)}.
\end{align*}
We remind here that $R$ is a probability measure and the definition \eqref{pseudo-energy} of $\EE.$
This entails by Sobolev embedding that:
\begin{align*}
| \langle L_1[\mathbb K] , w \rangle | & \leq 
C_p \left(  \int_0^t \mathcal E + \dfrac{\eps^2}{\tau^2}   \right) 
\sup_{[0,t]} \|w\|_{W^{3,p}(\mathbb R^d)} \\
&\leq K_p
\left( (1+t)^{(1-\alpha)_+}  + \ln(1+t)\, \mathbf{1}_{\alpha=1} + \mathbf{1}_{\nu >0} \sqrt{1+t} \right) \sup_{[0,t]} \|w\|_{W^{3,p}(\mathbb R^d)},
\end{align*}
where we used that $\tau^{-1}$ decays like $1/(1+t)$ to integrate $1/\tau^{2}$.
Similarly, we apply the control induced by $\mathcal D$ to bound:
\begin{align*}
| \langle L_1[\mathbb S] , w \rangle | &  \leq  \int_{0}^t \dfrac{\nu}{\tau^2} \int_{\mathbb R^d}  \sqrt{R} \mathbb S: \nabla w \\
& \leq \sqrt{\nu}\int_0^{t} \left( \dfrac{\nu}{\tau^4}  \int_{\mathbb R^d}  |\mathbb T^s|^2 \right)^{1/2}\left(\int_{\mathbb R^d} R \right)^{1/2} \sup_{[0,t]} \|\nabla w\|_{L^{\infty}(\mathbb R^d)}\\
& \leq C_p \sqrt{\nu t} \left( \int_0^{\infty} \mathcal D\right)^{\frac 12} 
\sup_{[0,t]} \|w\|_{W^{2,p}(\mathbb R^d)}.
\end{align*}
Combining the previous two estimates yields finally:
\begin{equation} \label{eq_estL1}
|\langle L_1,w \rangle| \leq K_p \left( (1+t)^{(1-\alpha)_+} +
  \ln(1+t)\, \mathbf{1}_{\alpha=1} + \mathbf{1}_{\nu >0} \sqrt{1+t}
\right) \sup_{[0,t]} \|w\|_{W^{3,p}(\mathbb R^d)}. 
\end{equation}
To handle $L_2,$ we use that $R$ has constant mass and the growth of $\tau$
at infinity:
\begin{align*}
|\langle L_2, w \rangle | & \leq \left| \int_0^t \dfrac{\nu \dot{\tau}}{\tau} \int_{\mathbb R^d} R \Div w\right| 
\leq  C \nu \int_0^t \dfrac{1}{\tau} \sup_{[0,t]}\|\Div w\|_{L^{\infty}(\mathbb R^d)}
\leq K_p \nu \ln(1+t) \sup_{[0,t]}\|w\|_{W^{2,p}}.
\end{align*}
We proceed with $L_3.$ First, we make controlled quantities appear via H\"older inequality:
\begin{align*}
|\langle L_3 , w \rangle| &\leq \left| \int_0^t \dfrac{1}{\tau^2} \int_{\mathbb R^d} \sqrt{R}U \otimes \sqrt{R}U : \nabla w \right| \\
& \leq  \int_0^t  \dfrac{1}{\tau^2} \int_{\mathbb R^d} |\sqrt{R}U|^2 
\, \sup_{[0,t]}  \|\nabla w\|_{L^{\infty}(\mathbb R^d)} \\
& \leq  \left(\int_0^t  \mathcal E \, \d s \right)
\, \sup_{[0,t]}  \|\nabla w\|_{L^{\infty}(\mathbb R^d)} \\
& \leq K_p \left( (1+t)^{(1-\alpha)_+} + \ln(1+t)\, \mathbf{1}_{\alpha=1} + \mathbf{1}_{\nu>0} \sqrt{1+t}   \right)
\, \sup_{[0,t]}  \|w\|_{W^{2,p}(\mathbb R^d)}.
\end{align*}
Concerning $L_4,$ we have
\begin{align*}
|\langle L_4,w \rangle| 
&\leq \frac12\int_0^t \dfrac{\alpha}{\tau^{\alpha}} \left( \int_{\mathbb R^d} R \right)^{1/2} \left( \int_{\mathbb R^d} R |y|^2 \right)^{1/2}  \sup_{[0,t]}\|w\|_{L^{\infty}(\mathbb R^d)} \\
& \leq \frac 12 \left(\int_{0}^{t} \dfrac{1}{\tau^{\alpha/2}} \sqrt{\mathcal E}\, \d s    \right) \sup_{[0,t]} \|w\|_{W^{1,p}(\mathbb R^d)} \\
& \leq K_p\left( \int_{0}^{t}  
\left( \dfrac{1}{(1+s)^{\alpha}} +  \dfrac{\nu (1+\ln(1+s))}{(1+s)^{(1+\alpha)/2}} \right)   \d s    \right) \sup_{[0,t]} \|w\|_{W^{1,p}(\mathbb R^d)} \\
& \leq K_p \left( 
(1+t)^{(1-\alpha)_+} + \ln(1+t)\mathbf{1}_{\alpha=1} +  \mathbf{1}_{\nu >0} \sqrt{1+t} 
   \right) \sup_{[0,t]}\|w\|_{W^{1,p}(\mathbb R^d)} .
\end{align*}

Finally, for $L_5,$ we obtain directly that:
\begin{align*}
|\langle L_5, w \rangle| 
&\leq (\gamma-1) \left(\int_0^t \mathcal E \, \d s \right) \sup_{[0,t]} \|\Div w\|_{L^{\infty}} \\ 
&\leq K_p \left( (1+t)^{(1-\alpha)_+}  + \ln(1+t)\, \mathbf{1}_{\alpha=1} + \mathbf{1}_{\nu >0} \sqrt{1+t} \right) \sup_{[0,t]}\|w\|_{W^{2,p}(\mathbb R^d)}.
\end{align*}

This completes the proof.
\end{proof}

We apply now this control of $RU$ in order to handle $\partial_t R.$
We have:

\begin{proposition}\label{prop:sob}
For arbitrary $\phi \in C^{\infty}_c(\mathbb R^d)$ the function 
\[
R_{\phi} : t \mapsto \int_{\mathbb R^d} R(t,y) \phi(y){\rm d}y 
\]
enjoys the properties:
\begin{itemize}
\item[i)] $R_{\phi} \in C([0,\infty))$,
\item[ii)] $R_{\phi}$ converges to some limit $R^{\infty}_{\phi}$ as $t \to \infty$, satisfying:
\[
|R^{\infty}_{\phi}| \leq C_{p,\infty} \|\phi\|_{W^{4,p}(\mathbb R^d)},
\]
for a constant $C_{p,\infty}$ depending on $p > d$, but independent of $\phi.$
\end{itemize}
\end{proposition}
This latter result shows the convergence of $(R(t,\cdot))_{t>0}$ through the mapping
$\phi \mapsto R_{\phi}^\infty$ in $W^{-{4,p'}}(\mathbb R^d).$ This mapping is bound to be a  probability measure thanks to the tightness of $(R(t,\cdot))_{t>0}$. Hence, the proof of this proposition ends up this part.
 
\begin{proof}
Similarly to the previous proof, we have here that $R \in L^{\infty}_{\rm loc}(0,\infty;L^{\gamma}(\mathbb R^d))$ and, thanks to Equation~\eqref{RU1} 
with (H2), there holds: $\partial_t R \in L^1_{\rm loc}(0,\infty;W^{-1,\gamma_*'}(\mathbb R^d)) $ (where $\gamma_*'$ is the conjugate exponent of $\gamma_*$). Applying {\bf Lemma \ref{lem_reg}}, we have then that, for arbitrary $\phi \in C^{\infty}_c(\mathbb R^d)$ and $t_1 < t_2$ there holds
$R_{\phi} \in C([0,\infty))$ and 
\[
R_{\phi}(t_2) - R_{\phi}(t_1) = \int_{t_1}^{t_2} \dfrac{1}{\tau^2} RU \cdot \nabla \phi .
\] 
In this equality, we apply the bound of {\bf Proposition \ref{prop_RU}} with 
$p>d.$ This yields:
\begin{multline*}
| R_{\phi}(t_2) - R_{\phi}(t_1)  | \\ \leq K_p \left( \int_{t_1}^{t_2} \dfrac{   (1+ t)^{(1-\alpha)_{+}} + \ln(1+t)\mathbf{1}_{\alpha=1} + \mathbf{1}_{\nu >0}(1+t)^{1/2}}{\tau(t)^2} \, {\rm d}t \right)  \|\nabla \phi\|_{W^{3,p}(\mathbb R^d)} .
\end{multline*}
Since $\tau \sim t$ for  large $t$, we obtain that 
\[
t \mapsto \dfrac{  (1+ t)^{(1-\alpha)_{+}} + \ln(1+t)\mathbf{1}_{\alpha=1} + \mathbf{1}_{\nu >0}(1+t)^{1/2} }{\tau(t)^2} \in L^1([0,\infty)).
\]
By a standard domination argument, we infer the conclusions of our proposition: $R_{\phi}$ admits a limit $R_{\phi}^{\infty}$ when $t \to \infty$, and 
\begin{align*}
|R_{\phi}^{\infty}| 
  &\leq \left| \int_{\mathbb R^d} R(0,\cdot) \phi \right|\\
  &\quad+ K_p \left( \int_0^{\infty} \dfrac{ (1+ t)^{(1-\alpha)_{+}} + \ln(1+t)\mathbf{1}_{\alpha=1} + \mathbf{1}_{\nu >0}(1+t)^{1/2}}{\tau(t)^2} \, {\rm d}t \right) \|\phi\|_{W^{4,p}(\mathbb R^d)} \\
&\leq C_{p,\infty} \|\phi\|_{W^{4,p}(\mathbb R^d)}.
\end{align*}
\end{proof} 

As a straightforward corollary to the above computations, 
we also have the following convergence result for any $p >d$: 
\[
\|R(t,\cdot) - R_{\infty}\|_{W^{-4,p'}(\mathbb R^d)} \leq 
 K_p \left( \dfrac{1}{(1+t)^{\min(\alpha,1)}} + \dfrac{\ln(1+t)}{(1+t)} \mathbf{1}_{\alpha =1} + \dfrac{\mathbf{1}_{\nu>0}}{\sqrt{1+t}}\right), \quad \forall \, t >0.
\]

Furthermore, for sufficiently small $\gamma,$ we can also state more properties of the asymptotic $R_{\infty}.$ Indeed, from (H4) we infer that:
\[
\int_{\mathbb R^d} R^{\gamma}(t,\cdot) \leq C_0  \tau^{d(\gamma-1)} \left( \dfrac{1}{(1+t)^{\alpha}} + \dfrac{\nu}{(1+t)} \left( \mathbf{1}_{\alpha\neq 1} + \ln(1+t)\mathbf{1}_{\alpha=1} \right)\right) \quad \forall \, t >0.
\]
Consequently, when $\nu >0,$ if $d(\gamma-1) \leq 1$ ({\em i.e.} $\gamma < 1+1/d$ and $\alpha =d(\gamma-1)$)  we obtain that $(R(t,\cdot))_{t>0}$ is bounded in $L^{\gamma}(\mathbb R^d).$ 
While, when $\nu=0,$ the same holds true for  $d(\gamma-1) \leq 2$ {\em i.e.} $\gamma \leq 1+ 2/d.$ In both cases, the uniform $L^{\gamma}$-bound ensures that the asymptotic profile
$R_{\infty}$ is not only a probability measure but also an
$L^1$-function.
Finally, when $\eps >0$, $\nu =0$ and $\gamma \ge 1 + 2/d$, we have $\alpha = 2$ and from (H4) we obtain
$$
\frac{\eps^2}{2}\int_{\R^d} |\nabla \sqrt R(t,\cdot)|^2
\le C_0  \dfrac{\tau^2(t)}{(1+t)^{2}}  \quad \forall \, t >0,
$$
which implies that $(\nabla \sqrt R(t,\cdot))_{t>0}$ is bounded in $L^{2}(\mathbb R^d)$ and thus that $R_\infty \in L^1(\R^d)$ also in this case.


\section{Nonlinear Schr\"odinger equation}
\label{sec:nls}

\subsection{A priori estimates}

For $\tau$ solution to \eqref{eq:tau}, and  $\psi^\eps$ solution to
\eqref{eq:nls} with $\psi_0^\eps\in \Sigma$ as defined in
Theorem~\ref{theo:NLS}, $\Psi^\eps$ given by 
\eqref{eq:psiPsi} solves
\begin{equation}
  \label{eq:Psi}
  i\eps\partial_t \Psi^\eps + \frac{\eps^2}{2\tau(t)^2}\Delta \Psi^\eps =
  \frac{\alpha}{\tau(t)^{2\alpha} }\frac{|y|^2}{2}\Psi^\eps
    +\frac{\mu^\eps}{\tau(t)^{d\si}}|\Psi^\eps|^{2\si}\Psi^\eps,\quad
    \Psi^\eps_{\mid t=0} = \frac{\psi_0^\eps}{\|\psi_0^\eps\|_{L^2(\R^d)}},
\end{equation}
where
\begin{equation*}
  \mu^\eps = \lambda \|\psi_0^\eps\|_{L^2(\R^d)}^{2\si}. 
\end{equation*}
The pseudo-energy for $\Psi^\eps$ is
\begin{equation}\label{eq:energyPsi}
\begin{aligned}
  \EE^\eps(\Psi^\eps) &= \frac{\eps^2}{2\tau(t)^2}\|\nabla
  \Psi^\eps(t)\|_{L^2(\R^d)}^2 +
  \frac{\alpha}{2\tau(t)^{2\alpha}}\int_{\R^d} |y|^2
                        |\Psi^\eps(t,y)|^2\dy\\
  &\quad+\frac{\mu^\eps}{(\si+1)\tau(t)^{d\si}} \int_{\R^d}
  |\Psi^\eps(t,y)|^{2\si+2}\dy ,
\end{aligned}
\end{equation}
and satisfies
\begin{equation*}
  \frac{\d}{\d t}\EE^\eps(\Psi^\eps) +\DD^\eps(\Psi^\eps)=0,
\end{equation*}
where the dissipation is given by
\begin{equation}\label{eq:dissipPsi}
  \begin{aligned}
     \DD^\eps(\Psi^\eps) = \frac{\dot \tau}{\tau}\Bigg( & \frac{\eps^2}{\tau(t)^2}\|\nabla
  \Psi^\eps(t)\|_{L^2(\R^d)}^2 +
  \frac{\alpha^2}{\tau(t)^{2\alpha}}\int_{\R^d} |y|^2
                                                      |\Psi^\eps(t,y)|^2\dy\\
  & +\frac{d\si \mu^\eps}{(\si+1)\tau(t)^{d\si}} \int_{\R^d}
  |\Psi^\eps(t,y)|^{2\si+2}\dy \Bigg).
  \end{aligned}
 \end{equation}
In the case of the nonlinear Schr\"odinger equation, justifying the
above identity is standard at the level of regularity that we consider, and we refer to \cite{CazCourant} for
details. We infer:
\begin{proposition}\label{prop:nls}
    Let $d\ge 1$, $\eps,\lambda>0$, $0<\si<\frac{2}{(d-2)_+}$, and
    $\psi_0^\eps\in \Sigma$. Then for $\tau$ solution to
    \eqref{eq:tau} with $\alpha = \min \(\tfrac{d\si}{2},1\)$, the
    function $\Psi^\eps$ defined by \eqref{eq:psiPsi}  satisfies
    \begin{equation*}
 \EE^\eps(\Psi^\eps(t)) \le \frac{
   \EE^\eps(\Psi^\eps(0))}{\tau(t)^{\min(2,d\si)}},\quad \forall t\ge
 0,\quad \int_0^\infty  \DD^\eps(\Psi^\eps(t))\d t<\infty,
\end{equation*}
where $\EE^\eps$ is given by \eqref{eq:energyPsi} and $\DD^\eps$ is
given by \eqref{eq:dissipPsi}.
\end{proposition}

The above proposition provides the same a priori estimates as we have
used in the case of the Euler-Korteweg system. More precisely, for $(\ro,u)$ related
to $\psi^\eps$ thanks to Madelung transform like in
\eqref{eq:madelung}, we note that Madelung transform for
$\Psi^\eps$ provides
\begin{equation*}
  R^\eps = |\Psi^\eps|^2,\quad R^\eps U^\eps = \eps\IM \(\bar
  \Psi^\eps\nabla\Psi^\eps\), 
\end{equation*}
and thus, $(\ro,u)$ and $(R^\eps,U^\eps)$ are related through
\begin{equation*}
  \ro(t,x)=\frac{1}{\tau(t)^d}R^\eps\(t,\frac{x}{\tau(t)}\),\quad
  u(t,x) = \frac{1}{\tau(t)}U^\eps\(t,\frac{x}{\tau(t)}\)+\frac{\dot
    \tau(t)}{\tau(t)} x,
\end{equation*}
which is exactly \eqref{eq:uvFluid}. Theorem~\ref{theo:NLS} then
appears as a direct consequence of Theorem~\ref{theo:rigidity} in the
Euler-Korteweg case.
\subsection{Interpretation}

We now comment on some consequences of
Proposition~\ref{prop:nls}.

\subsubsection{Long range scattering}
Suppose $\eps=1$. In the case $\si\ge 2/d$, a complete scattering
theory is available for \eqref{eq:nls}, in the sense that given
$\psi_0\in \Sigma$, there exists $\psi_+\in \Sigma$ such that
\begin{equation*}
  \left\|e^{-i\frac{t}{2}\Delta}\psi(t)-\psi_+\right\|_{\Sigma}\Tend t \infty
  0,\quad \|f\|_{\Sigma}^2:= \|f\|_{L^2(\R^d)}^2 + \|\nabla
    f\|_{L^2(\R^d)}^2 +\|x f\|_{L^2(\R^d)}^2 .
  \end{equation*}
  As a matter of fact, the same is true under the weaker assumption
  $\si\ge \si_0(d)$ for some $1/d<\si_0(d)<2/d$; see
  e.g. \cite{CazCourant}. A weaker convergence (in $L^2(\R^d)$ instead
  of $\Sigma$, with $\psi_+\in L^2(\R^d)$, \cite{TsYa84}, and even
  $\psi_+\in H^1(\R^d)$, \cite{GinibreDEA}) holds for $\si>1/d$. For
  $0<\si\le 1/d$, long range effects are present, as evoked in the
  introduction. In the case of \eqref{eq:nls}, the long range effects
  are understood only in the critical case $\si=1/d$: see \cite{HN06}
  and references therein. See also \cite{MMU19} and references therein
  for  the existence of 
  wave operators (Cauchy problem with prescribed behavior at
  $t=\infty$ instead of $t=0$) in the case $\si=1/d$.
It seems that so far, the long
  range scattering has not been studied for \eqref{eq:nls} in
  the case $0<\si<1/d$. The lack of regularity of the nonlinearity is
  an important technical difficulty, which was bypassed in the
  analogous case of (generalized) Hartree nonlinearities, see
\cite{GV00a,GV00b,GV01} and references therein.
\smallbreak

  In the case $0<\si\le 1/d$, Proposition~\ref{prop:nls} yields, for
  $\eps=1$,
  \begin{equation*}
    \int_{\R^d} |y|^2|\Psi(t,y)|^2\dy +
    \|\Psi(t)\|_{L^{2\si+2}(\R^d)}^{2\si+2}\lesssim 1,\quad \forall
    t\ge 0,
  \end{equation*}
and Theorem~\ref{theo:NLS} shows the convergence of
$|\Psi(t,\cdot)|^2$ in the limit $t\to \infty$, indicating that for
the full range $0<\si\le 1/d$, long range effects do not affect the
dispersive behavior, and present at leading order only in a phase modification.

\subsubsection{Semi-classical limit}
Consider the limit $\eps\to 0$ in \eqref{eq:nls}, for initial data
under a WKB form,
\begin{equation*}
  \psi^\eps_0(x) = a_0(x)e^{i\phi_0(x)/\eps},
\end{equation*}
with $a_0$ and $\phi_0$ smooth and independent of $\eps$, $\phi_0$
being real-valued. In particular, the $L^2$-norm of $\psi_0^\eps$ is
independent of $\eps$. 
Proposition~\ref{prop:nls} then yields, in the case $0<\si\le 1/d$,
\begin{equation}\label{eq:nocaustic}
    \int_{\R^d} |y|^2|\Psi^\eps(t,y)|^2\dy +
    \|\Psi^\eps(t)\|_{L^{2\si+2}(\R^d)}^{2\si+2}\le C_0,\quad \forall
    t\ge 0,
  \end{equation}
  for some $C_0>0$ independent of $\eps$. This estimate indicates
  dispersive properties which are uniform in $\eps$, a phenomenon
  which cannot hold in the linear case
  \begin{equation*}
    i\eps\partial_t \psi^\eps_{\rm lin} + \frac{\eps^2}{2}\Delta
    \psi^\eps_{\rm lin} = V(x)\psi^\eps_{\rm lin} ,\quad
    \psi^\eps_0(x) = a_0(x)e^{i\phi_0(x)/\eps}, 
  \end{equation*}
where the formation of caustics is incompatible with
\eqref{eq:nocaustic}.
Indeed in the linear case, the rapid oscillation are described,
initially, by a Hamilton-Jacobi equation, whose solution may become 
singular in finite time, precisely on the caustic set: this
geometrical phenomenon coincides with the amplification of 
the order of magnitude of $\psi_{\rm lin}^\eps$ in the limit $\eps\to
0$.
While the $L^2$-norm of $\psi_{\rm lin}^\eps$ is independent of (time
and ) $\eps$, its $L^{2\si+2}$-norm blows up like some negative power
of $\eps$. The precise value of this power depends on the geometry of
the caustic (see e.g. \cite{Du,DuBook,JMRMemoir}), and is for instance equal to $-\frac{d\si}{2\si+2} $ in
the case of a focal point (see e.g. \cite{CaBook2}).
In the
case of \eqref{eq:nls}, the Hamilton-Jacobi equation is replaced
 by a compressible Euler equation (\eqref{fluide1}-\eqref{fluide2}
 with $\eps=\nu=0$, where $\lambda$, $\sigma$ and $\gamma$ are related like
 in \eqref{eq:madelung}), whose solution may
develop singularities in finite time,  from \cite{MUK86}. However,
there is no amplification of $\Psi^\eps$, at least 
in $L^2\cap L^{2\si+2}$. This suggests that the notion of caustic must
be adapted in this case, for the geometrical phenomenon and the
analytical phenomenon, which coincide in the linear case, no longer do:
the nonlinearity in \eqref{eq:nls} 
prevents the amplification phenomenon.

\subsection{Proof of Proposition~\ref{prop:wave-op-korteweg}}
\label{sec:wave-op-korteweg}
Proposition~\ref{prop:wave-op-korteweg} is actually valid for more general
profiles, $a_\infty\in
\Sigma$. 
Let $\eps>0=\nu $, $\gamma$ like in
Proposition~\ref{prop:wave-op-korteweg}, and $a_\infty\in \Sigma $. Define
$\psi_+^\eps$ by
\begin{equation*}
  a_\infty(x)=\frac{1}{\eps^{d/2}}\hat
  \psi_+^\eps\(\frac{x}{\eps}\),\quad\text{where} \quad \hat f(\xi)=
  \frac{1}{(2\pi)^{d/2}} \int_{\R^d} e^{-ix\cdot\xi}f(x)\dx, \quad
  f\in \mathcal S(\R^d).
\end{equation*}
Since $\Sigma = H^1\cap \mathcal F(H^1)$, $\psi^\eps_+\in \Sigma$. 
Standard
scattering theory for NLS (see e.g. \cite{CazCourant,GinibreDEA})
implies that there exists a unique solution $\psi^\eps\in
C(\R_+;\Sigma)\cap L^{\frac{4\si+4}{d\si}}(\R_+;L^{2\si+2}(\R^d))$ to
\eqref{eq:nls}, with
\begin{equation*}
  \lambda =\frac{\gamma}{\gamma-1}>0,\quad \si =\frac{\gamma-1}{2},
\end{equation*}
such that
\begin{equation*}
  \|e^{-i\eps\frac{t}{2}\Delta} \psi^\eps(t)-\psi_+\|_{\Sigma}\Tend t
    \infty 0.
  \end{equation*}
  Since $e^{i\eps\frac{t}{2}\Delta} $ is unitary on $L^2(\R^d)$, this
  implies
  \begin{equation*}
    \| \psi^\eps(t)-e^{i\eps\frac{t}{2}\Delta}\psi_+\|_{L^2(\R^d)}\Tend t
    \infty 0.
  \end{equation*}
  On the other hand (see e.g. \cite{Tsutsumi85}),
   \begin{align*}
    \| e^{i\eps\frac{t}{2}\Delta}\psi^\eps_+-A(\psi^\eps_+)(t)\|_{L^2(\R^d)}\Tend t
    \infty 0,\quad \text{where}\quad
          A(\psi_+)(t,x) &=
    \frac{1}{(i\eps t)^{d/2}}\hat
    \psi_+^\eps\(\frac{x}{\eps t}\)e^{i\frac{|x|^2}{2\eps
        t}}\\
    &=\frac{1}{(i t)^{d/2}}a_\infty\(\frac{x}{ t}\)e^{i\frac{|x|^2}{2\eps t} }.
    \end{align*}
  We infer, from Cauchy--Schwarz and triangle inequalities,
    \begin{align*}
      \left\| |\psi^\eps(t)|^2 -
      \frac{1}{t^d}\left|a_\infty\(\frac{x}{t}\)\right|^2\right\|_{L^1(\R^d)}&=\left\|
       |\psi^\eps(t)|^2   - 
      \left|A(\psi_+^\eps)(t)\right|^2\right\|_{L^1(\R^d)}\\
      &\le
  \( \left\| \psi^\eps (t)\right\|_{L^2(\R^d)} +
      \left\|  A(\psi_+^\eps)(t)\right\|_{L^2(\R^d)}  \)  \left\| \psi^\eps (t) -
 A(\psi_+^\eps)(t)\right\|_{L^2(\R^d)}\\
      &\le 2\|\psi_+^\eps\|_{L^2(\R^d)} \left\| \psi^\eps (t) -
 A(\psi_+^\eps)(t)\right\|_{L^2(\R^d)}\Tend t \infty 0,
    \end{align*}
hence Proposition~\ref{prop:wave-op-korteweg} by defining $(\ro,u)$ by
Madelung transform \eqref{eq:madelung}, so it solves the
Euler-Korteweg system.


\section{Proof of Theorem~\ref{theo:existence}}
\label{sec:existence}

We end the paper with the proof of {\bf Theorem~\ref{theo:existence}}. 
This section is split into three subsections corresponding to the three different cases 
in {\bf Theorem~\ref{theo:existence}}.
\subsection{Euler}

The first case of Theorem~\ref{theo:existence} is simply a
reformulation of the main result from \cite{Gra98}. The assumption
made on $u_0$ ensures that the (multidimensional) Burgers equation
\begin{equation*}
  \partial_t \bar u +\bar u\cdot \nabla \bar u=0,\quad \bar u_{\mid
    t=0}= u_0,
\end{equation*}
has a unique, global solution for $t\ge 0$. A typical example is
$u_0(x)=x$. Then \cite[Theorem~1]{Gra98} asserts that the Euler
equation \eqref{fluide1}-\eqref{fluide2} has a global smooth solution such that
\begin{equation*}
  \ro^{\frac{\gamma-1}{2}},u-\bar u\in
  C^j([0,\infty[;H^{s-j}(\R^d)),\quad j=0,1.
\end{equation*}
At this level of
regularity, the unknowns $(R,U)$ obtained from $(\ro,u)$ through the change of unknown
\eqref{eq:uvFluid} satisfy a fortiori (H1)--(H3). Moreover, when $\alpha = \min(2,d(\gamma-1))$  all the formal manipulations
leading to \eqref{eq_dEEt0}-\eqref{RU:decay} and \eqref{Ry2} are rigorously justified, and so (H4) is
satisfied too.

\begin{remark}
  In \cite{Serre97}, the assumption on $\gamma$ is restricted to
  $1<\gamma\le 1+2/d$, $\ro_0$ need not be compactly supported, and the
  assumption on $u_0$ reads $v_0\in H^s(\R^d)$ with
  $\|v_0\|_{H^s(\R^)}\ll 1$, where $v_0(x)=u_0(x)-x$. The conclusion
  is then the same as above, with $\bar u$ replaced by
  \begin{equation*}
    \bar u(t,x) = \frac{x}{t+1},
  \end{equation*}
  which is a particular solution of the Burgers equation.
  Therefore, 
  the assumption on $\ro_0$ is slightly weaker, but the assumption on
  $u_0$ appears to be a particular case of the framework considered in
  \cite{Gra98}. 
\end{remark}
\subsection{Euler-Korteweg}
The second case of Theorem~\ref{theo:existence} is a consequence of
Madelung transform \eqref{eq:madelung} and of the identities presented
in Section~\ref{sec:nls}. Indeed the assumption
$1<\gamma<1+\frac{4}{(d-2)_+}$ from Theorem~\ref{theo:existence}
corresponds to $0<\si <\frac{2}{(d-2)_+}$ in \eqref{eq:nls}, with
$\lambda>0$ (referred to as defocusing case). Since we assume
$\psi_0\in \Sigma$, standard Cauchy theory for  \eqref{eq:nls} (see
e.g. \cite{CazCourant}) yields the existence of a unique solution
 \begin{equation*}
    \psi^\eps \in C(\R;\Sigma)\cap L^{\frac{4\si+4}{d\si}}_{\rm
      loc}(\R;L^{2\si+2}(\R^d)). 
  \end{equation*}
  In particular, we also have
  \begin{equation*}
    \Psi^\eps \in C(\R;\Sigma)\cap L^{\frac{4\si+4}{d\si}}_{\rm
      loc}(\R;L^{2\si+2}(\R^d)). 
  \end{equation*}
  As noticed in Section~\ref{sec:nls}, $R=|\Psi^\eps|^2$ and
  $RU=\eps\IM(\bar\Psi^\eps\nabla \Psi^\eps)$, and (H1)--(H3) are
  satisfied. In addition (see \cite{AnMa09} or \cite{CaDaSa12}), 
  \begin{equation*}
   \eps^2 |\nabla\Psi^\eps|^2 = |\eps\nabla \sqrt R|^2+ R|U|^2.
  \end{equation*}
  Therefore  \eqref{eq:energyPsi} corresponds exactly to
  \eqref{pseudo-energy}, and {\bf Proposition~\ref{prop:nls}} shows that
  (H4) is satisfied.

\subsection{Navier-Stokes}

We proceed with case (iii) of {\bf Theorem \ref{theo:existence}} and
consider the Navier-Stokes system with capital-letter unknowns
\eqref{RU1}-\eqref{RU2} in dimension $d \leq 3,$ where $\nu > 0$ and
$\varepsilon \geq 0.$ 
The weak regularity statements of (H1) are not sufficient to prove the decay estimates of (H4) with multiplier arguments as in the introduction. We need to obtain these estimates ``by construction'' so that we provide here a brief proof of existence of weak solutions satisfying (H1)-(H4).
We note  that, up to time-dependent scaling
terms, 
\eqref{RU1}-\eqref{RU2}
is similar to the barotropic (quantum)
Navier-Stokes system. Construction of weak solution to this
system goes back to \cite{Lio98,FNP01} (Newtonian Navier-Stokes, i.e.\ non-degenerate viscosity), to \cite{Bre-De-CKL-03,Jungel} (degenerate viscosity) in a framework allowing the presence of vacuum thanks to a suitable setting of the problem, and is by now well-documented, see e.g.\
\cite{VasseurYuInventiones,LacroixVasseur,CCH2,AntonelliHientzschSpirito} and references therein
(see also \cite{RoussetBBK}). We emphasize that one of the
specificities of \cite{CCH2} is to construct solutions on $\R^d$,
instead of $\T^d$ like in most of the references. This approach has
been used in  \cite{AntonelliHientzschSpirito}  in the polytropic
case on $\R^d$. 
Our proof --- that weak solutions to \eqref{RU1}-\eqref{RU2} satisfying (H1)-(H4) do exist --- reproduces the strategy  of \cite{CCH2} (isothermal case on $\R^d$), 
which in turn is based on the approach of \cite{VasseurYuInventiones,LacroixVasseur} (polytropic case on the torus).
As a consequence, we do not give precise details. We only perform the formal energy estimates justifying our definition of weak solutions, give a scheme of the proof and explain how (H4) is achieved.

\subsubsection{Definition of weak solutions}
The system \eqref{RU1}-\eqref{RU2} is classically endowed with conservation \eqref{massR}
and dissipation estimate \eqref{dEEdt}. Such estimates are not sufficient to build up a satisfactory weak solution theory. These pieces of information are complemented with the decay of the by-now
called ``BD-entropy'' (see \cite{BD07}, among others). To construct
this new quantity, we differentiate \eqref{RU1} with respect to
space, and find:
\begin{equation} \label{RU3}
 \partial_t \(R \nabla \ln R\) + \dfrac{1}{\tau^2} {\rm div} (R \nabla \ln (R) \otimes U) + \dfrac{1}{\tau^2} {\rm div} (R \nabla U) = 0.
\end{equation}
The key-remark from \cite{BD07} here is that  the last term in this equation may combine with the Newtonian tensor in the moment equation \eqref{RU2}. So, we multiply \eqref{RU3} by $\nu$ and 
combine with \eqref{RU2}. Denoting $V = U + \nu \nabla \ln(R)$, we obtain:
\begin{multline} \label{RU4}
\partial_t (RV) + \dfrac{1}{\tau^2}{\rm div}(R V \otimes U) + \dfrac{\alpha}{2\tau^{\alpha}} yR 
+ \dfrac{1}{\tau^{d(\gamma-1)}} \nabla R^{\gamma} 
\\
= \dfrac{1}{\tau^2} {\rm div} \left( \dfrac{\varepsilon^2}{2} \mathbb K[R] + \nu \sqrt{R} \mathcal A[R,U] \right) + \dfrac{\nu \dot{\tau}}{\tau} \nabla R,
\end{multline}
where $\mathcal A= \sqrt{R} \A U.$ We perform then a classical energy estimate on this new equation by multiplying with $V/\tau^2.$ The two first terms yield the time-derivative of 
the kinetic energy associated with $V.$ The other terms are integrating by parts by splitting
$V = U + \nu \nabla\ln (R)$ and remarking that, for symmetry reasons, we have:
\[
\mathcal A[R,U] : \nabla^2 \ln(R) = 0.
\]
Eventually, we obtain:
\[
\frac{\d}{\dt} \EE_{\mathrm{BD}}[R,U] + \DD_{\mathrm{BD}}[R,U] 
= \frac{\alpha \nu d}{ 2 \tau^{2+\alpha}} \int R + \frac{\nu \dot \tau}{ \tau^3} \int R \Div U ,
\]
where the BD-entropy is defined by
\begin{align*}
\EE_{\mathrm{BD}}[R,U] := \frac{1}{2 \tau^2} \int \left( R|U+\nu \nabla \ln R|^2 + \eps^2 |\nabla \sqrt R|^2 \right) 
+\frac{\alpha}{4\tau^\alpha} \int |y|^2 R \\
+\frac{1}{(\gamma-1) \tau^{d(\gamma-1)}} \int R^\gamma ,
\end{align*}
and the associated nonnegative dissipation is given by
\begin{align*}
\DD_{\mathrm{BD}}[R,U] := & \dfrac{\dot{\tau}}{\tau}
\Biggl[ \frac{1}{\tau^2} \int \left( R|U|^2 + \eps^2 |\nabla \sqrt R|^2 \right) 
+\frac{\alpha^2}{4\tau^\alpha}  \int |y|^2 R
+\frac{d}{\tau^{d(\gamma-1)}} \int R^\gamma  
\Biggr],
\\ 
&+  \frac{\nu}{\tau^4} \int R | \A U|^2
+\frac{\nu \eps^2}{\tau^4} \int R |\nabla^2 \ln R|^2
+ \frac{4\nu}{\tau^{d(\gamma-1) + 2}} \int |\nabla R^{\gamma/2}|^2.
\end{align*}
With this further remark we can now set a definition of weak solution on the basis of all the
a priori bounded energy/entropy/dissipations:

\begin{definition}\label{def:weaksolutions}
Assume $\nu > 0,$ $\gamma > 1$  and $\eps \ge 0$. 
Let $(\sqrt{R_0} , \Lambda_0 = (\sqrt{R} U)_0 ) \in L^2(\R^d)  \times L^2(\R^d)$.
We call global weak
solution to \eqref{RU1}-\eqref{RU2}, associated to the initial data $(\sqrt{R_0} , \Lambda_0= (\sqrt{R} U)_0)$, any pair $(R,U)$ such that there exists a collection $(\sqrt{R},\sqrt{R}U,\mathbb K,\mathbb T)$ satisfying 
\begin{itemize}
\item[i)] The  following regularities:
\begin{align*}
&\(\<y\>+|U|\) \sqrt{R} \in L^{\infty}_{\rm loc}\(0,\infty; L^2 (\R^d)\),\quad \nabla
  \sqrt R \in L^{\infty}_{\rm loc}\(0,\infty; L^2 (\R^d)\),\\ 
 &  \sqrt{R} \in L^{\infty}_{\rm loc}(0,\infty ; L^{2\gamma}(\R^d)) \quad 
  \nabla {R}^{\gamma/2} \in L^2_{\rm loc}(0,\infty;L^2(\R^d)),
\\  
& \eps \nabla^2 \sqrt{R} \in L^2_{\rm loc}(0,\infty;L^2(\mathbb
  R^d)) ,\quad 
  \sqrt{\eps} \nabla R^{1/4}\in L^4_{\rm loc}(0,\infty;L^4(\mathbb
  R^d)),\\
&\mathbb T \in L^2_{\rm loc}(0,\infty; L^2(\mathbb R^d)) ,
\end{align*}
with the compatibility conditions
\[
\sqrt{R} \ge 0 \text{ a.e. on } (0,\infty)\times \R^d,  \quad   \sqrt{R}U=
0 \text{ a.e. on } \{\sqrt{R} = 0 \}.
\]
\item[ii)] The following equations in $\mathcal D'((0,\infty)\times \mathbb R^d)$
\begin{equation}\label{eq:NSKrevu}
  \left\{
    \begin{aligned}
  & \partial_t\sqrt{R}+\frac{1}{\tau^2}\Div (\sqrt{R} U )=
  \frac{1}{2\tau^2}{\rm Trace}(\mathbb T),\\
    &\partial_t ({R}U) +\frac{1}{\tau^2}\Div ( \sqrt{R}U \otimes \sqrt{R}U)
      +2 y |\sqrt{R}|^2 +\nabla \left( |\sqrt{R}|^2 \right)
      \\
      & \phantom{\partial_t (\sqrt{R} \sqrt{R}U) +\frac{1}{\tau^2}\Div }=\Div \left(
      \dfrac{\eps^2}{2\tau^2} \mathbb K  + \dfrac{\nu}{\tau^2} \sqrt R 
\mathbb T^s
        \right) +
      \dfrac{\nu \dot{\tau}}{\tau} \nabla R,
    \end{aligned}
\right.
\end{equation}
with 
$\mathbb T^s$
the symmetric part of $\mathbb T$ and the compatibility conditions:
\begin{align} \label{eq_compnewton}
& \sqrt{R}\mathbb T = \nabla(\sqrt{R}  \sqrt{R} U) - 2 \sqrt{R}U
  \otimes \nabla \sqrt{R}\,, \\[6pt] 
& \mathbb K 
=\sqrt{R}\nabla^2 \sqrt{R} -  \nabla \sqrt{R} \otimes \nabla \sqrt{R}
  \,. \label{eq_compkorteweg} 
\end{align}
\item[iii)] For any $\psi\in C_0^\infty(\R^d)$, 
  \begin{align*}
    &\lim_{t\to 0}\int_{\R^d} \sqrt{R}(t,y)\psi(y) \, \dd y=
    \int_{\R^d} \sqrt{R_0} (y) \psi(y)  \, \dd y , \\
& \lim_{t\to 0}\int_{\R^d} \sqrt{R}(t,y)    (\sqrt{R}U)(t,y)\psi(y) \, \dd y=
    \int_{\R^d} \sqrt{R_0} (y) \Lambda_0(y) \psi(y) \, \dd y.
  \end{align*}
\end{itemize}
\end{definition}

We point out that this definition is readily adapted from \cite[Definition 1.1]{CCH2}, where the isothermal case $\gamma=1$ is considered.
It is also similar to \cite[Definition 2.1]{AntonelliHientzschSpirito}.
The third existence statement in {\bf Theorem \ref{theo:existence}} is then 
a straightforward consequence of:

\begin{proposition} \label{prop:existence}
  Assume $\nu>0,$ $\gamma >1$ and $\eps\ge 0$.
   Let $(\sqrt{R_0} , \Lambda_0 = (\sqrt{R} U)_0 ) \in L^2(\R^d) \times L^2(\R^d)$ satisfy  the compatibility conditions
\[
\sqrt{R_0} \ge 0 \text{ a.e.\ on } \R^d,  \quad   (\sqrt{R}U)_0=
0 \text{ a.e.\ on } \{\sqrt{R_0} = 0 \},
\]
as well as $\EE[R_0,U_0] < \infty$, $\EE_{\mathrm{BD}}[R_0,U_0] < \infty$.
There exists  at least one global weak solution to \eqref{RU1}-\eqref{RU2} 
in the sense of Definition~\ref{def:weaksolutions},
which satisfies
(H1)-(H4) and the conservation of mass.
\end{proposition}

We stress that, in this definition, we set:
\[
R_0 = \sqrt{R_0}^2, \qquad U_0 = \dfrac{(\sqrt{R}U)_0}{\sqrt{R}_0} \mathbf{1}_{\sqrt{R}_0 >0}.
\]
This is the common way to define functions of $R$ and $U$ in such a framework. In particular, the
definition of the velocity-field $U$ is satisfactory since we enforce the condition $\sqrt{R}U = 0$
under the condition $\sqrt{R} =0$ in our construction and assumptions.

\subsubsection{Condition (H4) and roadmap of the proof of Proposition \ref{prop:existence}}

The proof of Proposition~\ref{prop:existence} follows the compactness
approach of \cite{CCH2} relying on the key-ingredients
introduced in \cite{BD06b}, and resumed in, e.g., 
\cite{Gis-VV15,VasseurYuInventiones,LacroixVasseur}. We point out that
one important novelty of \cite{CCH2} was to treat the isothermal case while the ingredients of \cite{VasseurYuInventiones,LacroixVasseur} handle the precise polytropic case that we consider herein. 
Consequently, we only point out the roadmap of the proof herein and refer the reader to these previous references for more details on the different ingredients, and how to combine them.

The first step of the proof consists in solving a regularized version of \eqref{RU1}-\eqref{RU2}
on a torus of arbitrary size $\ell >1$, denoted by $\T^d_\ell$. This regularized version is associated with parameters
$r= (r_0,r_1) \in (0,\infty)^2$, $\delta := (\delta_1,\delta_2) \in (0,\infty)^2$,  $(\eta_1,\eta_2) \in (0,\infty)^2$, 
$m>0$ sufficiently large (see \cite{VasseurYu}),
and involves a ``cold-pressure'' exponent $k \in (0,\infty)$ that has to be chosen sufficiently large. This regularized system reads:
\begin{align}
& \partial_t R+\frac{1}{\tau^2}\Div (R U )= \frac{\delta_1}{\tau^2} \Delta R,    \label{RU-drag-reg1}  \\
& \partial_t (R U) +\frac{1}{\tau^2}\Div ( R U \otimes U) +\dfrac{\alpha}{2\tau^{\alpha}}y R  +\dfrac{1}{\tau^{d(\gamma-1)}} \nabla P_c(R) \label{RU-drag-reg2} \\ 
&\quad\qquad 
+ \frac{r_0}{\tau^2} U + \frac{r_1}{\tau^2} R |U|^2 U + \frac{\delta_1}{\tau^2}(\nabla R \cdot \nabla) U  \nonumber \\\notag
&\quad\qquad
=\frac{\eps^2}{2\tau^2}R\nabla \( \frac{\Delta  \sqrt{R}}{\sqrt{R}} \)   
+\frac{\nu}{\tau^2} \Div (R \D U) + \frac{\nu \dot \tau}{\tau} \nabla R
+\frac{\delta_2}{\tau^2} \Delta^2 U  + \frac{\eta_2}{\tau^2} R \nabla
   \Delta^{2m+1} R  ,
\end{align}
where:
\[
P_c(R) = R^{\gamma} - \dfrac{\eta_1}{R^{k}}.
\]  
By a suitable truncation/regularization of the initial condition,
the regularized system is solved when completed with initial data $(R_0,U_0)$ satisfying:
\begin{equation}\label{RUinitial-theta2}
 R_0 \in C^{\infty}(\mathbb T^d_{\ell}), \quad U_0 \in L^2(\mathbb T^d_{\ell}), \quad \inf_{y \in \T^d_\ell} R_0(y) \ge \theta >0 .
\end{equation}
The remaining steps of the analysis consist in letting successively $|\delta| \to 0,$ $|\eta| \to 0$
and  then $|r| \to 0,$ 
$\theta\to 0$, 
$\ell \to \infty$ (and possibly $\eps \to 0$).   
The following lemma ensures that assumption (H4) is satisfied at the level of the approximation:

\begin{lemma}
Given initial data $(R_0,U_0)$ satisfying \eqref{RUinitial-theta2}, there exists a global solution $(R,U)$ to \eqref{RU-drag-reg1}-\eqref{RU-drag-reg2} associated to $(R_0,U_0)$ on the torus $\T^d_\ell$, which satisfies moreover the  conservation of mass 
and the decay estimate (H4). 
\end{lemma}
The property (H4) being stable by weak convergence, the solution we construct inherits this property.

\begin{proof}
As in \cite[Section 2]{CCH2}, existence of solutions to \eqref{RU-drag-reg1}-\eqref{RU-drag-reg2} 
(with regularized initial data) is obtained {\em via} a Faedo-Galerkin approach. Namely,
the velocity-field $U$ is first chosen in a finite-dimensional
subspace of $L^2(\mathbb T^d_{\ell})$, 
Equation \eqref{RU-drag-reg2} being projected on this subspace, and \eqref{RU-drag-reg1} solved independently via a fixed-point argument. Again, we argue at the level of the finite-dimensional approximation, the same inequalities being satisfied by any limit of these approximations.

\medskip

Since the continuity equation \eqref{RU-drag-reg1} is satisfied pointwise, we have readily:
\[
\int_{\mathbb T^d_{\ell}} R(t,\cdot) = \int_{\mathbb T^d_{\ell}} R_0, \quad \forall \, t >0. 
\]
Here $R_0$ should be thought of as the regularized initial data, but
the regularization procedure ensures convergence of the mass of the
regularized approximation to the mass of $R_0.$ 
We obtain \eqref{massR}.

\medskip

At the level of the projection, all solutions are smooth in space and $C^1$ in time. Multiplying \eqref{RU-drag-reg2} by $U$ is then fully justified.  Similarly to the computation of dissipation estimate for the full system, 
we obtain (see also \cite[Proposition 2.6]{CCH2})  the following decay estimate:
\begin{equation} \label{eq_dissreg}
\frac{\d}{\dt} \EE_{\mathrm{reg}}[R,U] + \DD_{\mathrm{reg}}[R,U] = { \frac{\alpha d \delta_1}{ 2 \tau^{2+\alpha}}} \int R - \frac{\nu \dot \tau}{ \tau^3} \int R \Div U 
\end{equation}
where:
\[
\begin{aligned}
\EE_{\mathrm{reg}}[R,U] &=  \frac{1}{2\tau^2} \int_{\T^d_\ell} \left( R|U|^2 + \eps^2 |\nabla \sqrt{R}|^2 + \eta_2 \int_{\T^d_\ell} |\nabla \Delta^m R|^2 \right) \\
& \quad + \dfrac{\alpha}{4 \tau^{\alpha}} \int_{\T^d_\ell}  R|y|^2  + \dfrac{1}{\tau^{d(\gamma-1)}}  \int_{\T^d_{\ell}} \left( \dfrac{1}{\gamma-1}R^{\gamma}  + \dfrac{\eta_1}{k+1} \dfrac{1}{R^{k}}\right), \\
\end{aligned}
\]
and 
\[
\begin{aligned}
\DD_{\mathrm{reg}}[R,U] 
&= \frac{\dot \tau}{\tau}
\Biggl[ \dfrac{1}{\tau^2}\int_{\T^d_\ell}\left( R|U|^2 + \eps^2 |\nabla \sqrt{R}|^2+ \eta_2 |\nabla \Delta^m R|^2 \right)  \\
& \quad + \dfrac{\alpha^2}{4 \tau^{\alpha}} \int_{\T^d_\ell}  R|y|^2  + \dfrac{d(\gamma-1)}{\tau^{d(\gamma-1)}}  \int_{\T^d_{\ell}} \left( \dfrac{1}{\gamma-1}R^{\gamma}  + \dfrac{\eta_1}{k+1} \dfrac{1}{R^{k}}\right) \Biggr]  \\
&\quad + \frac{\nu}{\tau^4} \int_{\T^d_\ell} R |\D U|^2 
+\frac{\delta_2}{\tau^4}   \int_{\T^d_\ell} |\Delta U|^2
+ \frac{\delta_1 \eta_2}{\tau^4} \int_{\T^d_\ell} |\Delta^{m+1} R|^2 
\\
& \quad + \frac{\delta_1}{ \tau^{2+d(\gamma-1)}} \int_{\T^d_\ell} \left( \gamma R^{\gamma-2} + \dfrac{\eta_1 k}{R^{k+2}} \right) |\nabla R|^2  
\\
& \quad 
+ \frac{r_0}{\tau^4} \int_{\T^d_\ell} |U|^2
+ \frac{r_1}{\tau^4} \int_{\T^d_\ell}  R |U|^4 
+ \frac{\delta_1 \eps^2}{2\tau^4} \int_{\T^d_\ell} R| \nabla^2 \ln R|^2.
\end{aligned}
\]

At this point, we adapt the arguments of the introduction.
First we remark that the right-hand side $RHS$ of \eqref{eq_dissreg} satisfies:
\begin{equation} \label{eq_RHS}
RHS \leq \left( \dfrac{\alpha d\delta_1 }{\tau^{2 + \alpha}} + \nu \left( \dfrac{\dot \tau}{\tau}\right)^2\right) \int_{\mathbb T^{d}_{\ell}} R_0 + \dfrac{1}{2} \DD_{\mathrm{reg}}[R,U] .
\end{equation}
Again, here $R_0$ should be the regularized initial data but the
approximation procedure ensures convergence of the initial data in a
sense that is sufficient to guarantee that all constants involving
initial data are bounded by a constant depending on the initial data of the target system (see \cite[Section 4.2]{CCH2}). In particular, the
mass of the initial data can be assumed to be bounded by a constant $C_0$ depending only on initial data.

The inequality \eqref{eq_dissreg} then yields:
\[
\frac{\d}{\dt} \EE_{\mathrm{reg}}[R,U] 
+ \mathcal{D}_{\mathrm{reg}}[R,U] 
 \leq  C_0 \left( \dfrac{\alpha d \delta_1}{\tau^{2 + \alpha}} + \nu \left( \dfrac{\dot \tau}{\tau}\right)^2\right) .
\]
When $\delta_1 < \nu$, we can bound the right-hand side with a constant $C_{\alpha}$
depending only on $\alpha$:
\[
 \dfrac{\alpha d \delta_1}{\tau^{2 + \alpha}} + \nu \left( \dfrac{\dot \tau}{\tau}\right)^2
\leq \dfrac{C_{\alpha}\nu}{\tau^2} .
\]
Integration in time yields an $L^1(0,\infty)$-bound on $\mathcal{D}_{\mathrm{reg}}[R,U]$. We can then choose $\alpha = \min(2,d(\gamma-1))$ and argue as in the introduction
that:
\[
\DD_{\mathrm{reg}}[R,0] \geq \alpha \dfrac{\dot{\tau}}{\tau} \EE_{\mathrm{reg}}[R,U],
\] 
to  yield that $\EE_{\mathrm{reg}}[R,U]$ satisfies \eqref{eq_diss1}, and reproduce the computations of  \eqref{Ry2}.
We finally conclude that (H4) is satisfied by remarking that:
\[
\EE \leq \EE_{\mathrm{reg}}[R,U], \quad  \DD \leq \DD_{\mathrm{reg}}[R,U] .
\]
\end{proof}


\appendix

\section{Computations of formal energy estimates}
\label{sec:computations}

In this section, we consider $(\rho,u)$ a solution to \eqref{fluide1}-\eqref{fluide2} and justify
that at least formally, the decay estimate \eqref{rhou:decay} should be satisfied. Define the functional
\[
\begin{aligned}
A[\ro,u] 
&:=  t^2 E[\ro,u] - \int_{\R^d} t \ro u \cdot x 
+ \frac12 \int_{\R^d} \ro |x|^2 \\
&= \frac12 \int_{\R^d} \left( \ro \left|t u - {x}\right|^2 + t^2 \eps^2 |\nabla \sqrt \ro|^2 \right)  + \frac{t^2}{\gamma-1}\int_{\R^d} \ro^\gamma,
\end{aligned}
\]
where we recall that the energy $E[\ro,u]$ is defined in \eqref{energy}. A straightforward computation gives us
\[
\begin{aligned}
\frac{\d}{\dt} A[\ro,u] 
&= 2 t E[\ro,u] + t^2 \frac{\d}{\dt} E[\ro,u]
- t \int_{\R^d} \ro |u|^2 - t d \int_{\R^d} \ro^\gamma  \\
&\quad
-t {\eps^2} \int_{\R^d} |\nabla \sqrt \ro |^2 
+t \nu \int_{\R^d} \ro \Div u.
\end{aligned}
\]
Thanks to \eqref{dEdt} we then get
\[
\begin{aligned}
\frac{\d}{\dt} A[\ro,u] 
&= \frac{t}{\gamma-1} (2-d(\gamma-1)) \int_{\R^d} \ro^\gamma
-t^2 D[\ro,u] 
+t \nu \int_{\R^d} \ro \Div u.
\end{aligned}
\]
We now define the functional
\[
B[\ro,u] := \frac{1}{t^2} A[\ro,u]= \frac12 \int_{\R^d} \left( \ro \left|u - \frac{x}{t}\right|^2 + \eps^2 |\nabla \sqrt \ro|^2  \right)+ \frac{1}{\gamma-1}\int_{\R^d} \ro^\gamma 
\]
and we obtain, using previous computation, that
\[
\begin{aligned}
\frac{\d}{\dt}  B[\ro,u]  
&=  \frac{1}{t} \frac{(2-d(\gamma-1))}{\gamma-1} \int_{\R^d} \ro^\gamma
- D[\ro,u] 
+\frac{1}{t} \nu \int_{\R^d} \ro \Div u \\
&\quad
-\frac{2}{t^3} \frac12\int_{\R^d} \ro |tu-x|^2
-\frac{2}{t^3} \frac{t^2\eps^2}{2} \int_{\R^d} |\nabla \sqrt \ro|^2
-\frac{2}{t^3}\frac{t^2}{\gamma-1}\int_{\R^d} \ro^\gamma  \\
&= -\frac{d}{t}  \int_{\R^d} \ro^\gamma 
-\frac{1}{t}  \int_{\R^d} \ro \left|u - \frac{x}{t}\right|^2
-\frac{\eps^2}{t} \int_{\R^d} |\nabla \sqrt \ro |^2
- D[\ro,u]  
+\frac{1}{t} \nu \int_{\R^d} \ro \Div u  .
\end{aligned}
\]
We identify in last expression the terms in the definition of
$B[\ro,u]$, hence, using that $\frac{1}{t} \nu \int_{\R^d} \ro \Div
u\le \frac{\nu}{2t^2} \int_{\R^d} \ro + \frac{\nu}{2} \int_{\R^d} \ro
|\D u|^2$ and the conservation of mass, we deduce
\[
\begin{aligned}
\frac{\d}{\dt}  B[\ro,u] 
&\le -\frac{\min(2,d(\gamma-1))}{t} B[\ro,u]  + \frac{C\nu}{t^2} - \frac{1}{2}D[\ro,u] 
\end{aligned}
\]
for some constant $C>0$.
Therefore, for $t >0$, one has
\begin{equation}\label{rhou:decay_app}
B[\ro,u](t) \le \frac{C(E_0)}{(1+t)^{\min(2,d(\gamma-1))}} + \frac{C \nu }{1+t}.
\end{equation}

\section{Proof of Lemma~\ref{lem:tau}} \label{sec:tau}

\begin{proof}
  The local existence stems directly from the Cauchy-Lipschitz
  theorem. The only possible obstruction to the global propagation of
  regularity is the cancellation of $\tau$, which is impossible in
  view of the relation, obtained after multiplication of \eqref{eq:tau}
  by $\dot\tau$ and integration:
  \begin{equation*}
 \dot \tau(t)^2=
 1-\frac{1}{\tau(t)^{\alpha}}.
\end{equation*}
This implies $\tau(t)\ge 1$. 
Therefore, $\tau\in
C^\infty(\R;\R_+)$. The equation shows that $\tau$ is strictly (but not
uniformly) convex. If it was bounded, $\tau(t)\le M$, then we would
have
\begin{equation*}
  \ddot \tau(t)\ge \frac{\alpha}{2M^{1+\alpha}}>0,
\end{equation*}
hence a contradiction after two integrations. Therefore, $\tau(t_n)\to
\infty$ for some sequence $t_n\to \infty$, and since $\tau$ is convex,
$\tau(t)\to \infty$ as $t\to\infty$. Hence $\dot\tau(t)^2\to 1$, and since $\tau$ is
necessarily increasing for $t\ge 0$, $\dot\tau(t)\to 1$, and the
comparison of diverging integrals yields $\tau(t)\Eq t \infty t$. 
\end{proof}

\bibliographystyle{siam}
\bibliography{biblio}

\end{document}